\documentclass[11pt,reqno]{amsart}

\usepackage{amsmath,amsopn,amssymb,amsthm,multicol}

\usepackage{booktabs}
\usepackage{tabularx}
\usepackage{hyperref}
\usepackage{array}
\usepackage{graphicx}
\newcommand{\fr}{\mathfrak}

\newcommand{\op}{\operatorname}

 \newtheorem{lemma} {Lemma} [section]
\newtheorem{theorem}[lemma]{Theorem} 
\newtheorem{remark}[lemma] {Remark} 
\newtheorem{prop} [lemma]{Proposition}  
 
\newtheorem{corol}[lemma] {Corollary}

\newtheorem{claim}[lemma]{Claim}

\newtheorem{assum}[lemma]{Assumption}

\numberwithin{equation}{section}

\makeatletter
\def\ps@sourisheadings{%
  \def\@oddfoot{}%
  \def\@evenfoot{}%
  \def\@evenhead{%
    \normalfont\small
    \rlap{\thepage}\hfil
    NIKOLAOS PANAGIOTIS SOURIS\hfil}%
  \def\@oddhead{%
    \normalfont\small\hfil
    UNIFORM DOUBLING ON COMPACT HOMOGENEOUS SPACES
    \hfil\llap{\thepage}}%
}
\makeatother

\begin{document}

\title{Uniform doubling on compact homogeneous spaces with cyclic Lie brackets}
\author{Nikolaos Panagiotis Souris}
\address{University of Patras, Department of Mathematics, University Campus, 26504, Rio Patras, Greece}
\email{nsouris@upatras.gr}

\medskip

\begin{abstract}
We establish the uniform doubling property for $G$-invariant metrics on a wide class of compact Riemannian homogeneous spaces $(G/K,g)$, describing explicitly the volume growth of the metric balls $B_g(p,r)$. This property was previously known only for specific cases, including abelian Lie groups, the Lie group $SU(2)$ and quotients of $SU(2)\times \mathbb R^n$. Building on the approach of Eldredge, Gordina and Saloff-Coste for $SU(2)$, we refine and develop geometric and Lie-theoretic tools that allow us to replace the explicit identities of the Milnor basis in $SU(2)$ by a general structural assumption on the metric eigenspace decomposition. In particular, we show that the uniform doubling of $(G/K,g)$ is a consequence of the cyclic bracket condition $[\fr{m}_i,\fr{m}_j]=\fr{m}_k$, for $i,j,k$ pairwise distinct, a property that naturally generalizes the Milnor structure of $SU(2)$. We apply our results to $\mathbb Z_2\times \mathbb Z_2$-symmetric spaces, establishing the uniform doubling property for the complete family of $G$-invariant metrics on several classes of generalized Wallach spaces.  For compact homogeneous spaces, we also derive a global Poincar\'{e} inequality which holds uniformly for the spaces $(G/K,g)$ under consideration, with a constant controlled by the volume doubling constant of the space.



\end{abstract}

\maketitle

\pagestyle{sourisheadings}

\section{Introduction}
\subsection{Uniform doubling on Riemannian manifolds} An interesting question in Riemannian and metric geometry is whether the volume growth of a space can be controlled universally, no matter how much we deform its geometry in arbitrary directions. More precisely, a Riemannian manifold $(M,g)$ is called \emph{volume doubling} if 

\[ D_g:=\sup_{p\in M, \ r>0}\frac{\mu_g(B_g(p,2r))}{\mu_g(B_g(p,r))}<\infty, \]

\noindent where $\mu_g$ denotes the corresponding Riemannian volume measure and $B_g(p,r)$ denotes the metric ball with respect to the Riemannian distance $d_g$, centered at $p\in M$ with radius $r$. Accordingly, $D_g$ is called the \emph{volume doubling constant} of $(M,g)$. The crucial question is whether this property holds uniformly for entire families $\mathcal{M}$ of Riemannian metrics on $M$, i.e., whether there exists a constant $D$ such that $D_g\leq D$ for all $g\in \mathcal{M}$. If such a constant exists, we say that $M$ is \emph{uniformly doubling} on the family of metrics $\mathcal{M}$.

\subsection{Uniform doubling on Lie groups and the Milnor basis} In \cite{EGS18}, Eldredge, Gordina and Saloff-Coste showed that the Lie group $SU(2)$ is uniformly doubling on the complete family of its left-invariant metrics.

\begin{theorem}\label{EGSTheorem}\emph{(\cite{EGS18})} There exists a constant $D$ such that, for any left-invariant Riemannian metric $g$ on $SU(2)$, it holds $D_g\leq D$. \end{theorem}

\noindent  They also conjectured that any compact, connected Lie group $G$ has the same property. A positive answer to the conjecture would have special importance, as the existence of a uniform volume doubling constant yields universal analytic consequences, such as uniform Poincaré inequalities, Harnack inequalities, heat kernel estimates and spectral gap estimates (\cite{EGS18}).

In general, one cannot rely on curvature comparison to obtain uniform doubling for left-invariant metrics on compact Lie groups, as their Ricci curvature may not admit a uniform lower bound. In the case of $SU(2)$, Eldredge, Gordina and Saloff-Coste developed an alternative approach, constructing distinct intervals for the radius of the ball $B_g(e,r)$, bounded by suitable ratios of the square roots of the metric eigenvalues. As the radius crosses those intervals, the volume growth transitions from Euclidean, to sub-Riemannian, and then becomes comparable to the volume growth in $\mathbb S^2$, independently of the metric eigenvalues. In other words, it becomes uniformly comparable to $r^3,r^4,r^2$ respectively. Technically, the authors relied on the property that any left-invariant metric on $SU(2)$ can be diagonalized with respect to a \emph{Milnor basis} (\cite{Mil76}), i.e., a basis $\{e_1,e_2,e_3\}$ of the Lie algebra $\fr{su}(2)$ satisfying the cyclic Lie-bracket relations

\begin{equation*} [e_1,e_2]=e_3, \ \  [e_2,e_3]=e_1 \ \ \makebox{and} \ \ [e_3,e_1]=e_2. \end{equation*}

\noindent A key feature of the Milnor basis is that if any metric direction, say $e_3$, becomes arbitrarily expensive to reach directly, it can instead be reached through the cheaper directions $e_1,e_2$ via the  bracket $[e_1,e_2]=e_3$. Besides, the Milnor basis allows the use of explicit identities between the basis vectors, which were further developed to establish the uniform doubling property for quotients of $SU(2)\times \mathbb R^n$ (\cite{EGS25}). However, the authors stressed that their original proof is quite specific to the structure of $SU(2)$ and does not directly generalize to other Lie groups (\cite{EGS25}). 

\subsection{Main results: Cyclic Lie brackets imply uniform doubling on homogeneous spaces}
Inspired by the work of Eldredge, Gordina and Saloff-Coste, we study the uniform doubling property on compact homogeneous spaces $G/K$. Our arguments follow the general strategy introduced in their work, in which the volume growth of metric balls is studied over suitable intervals of the radius. However, we refine and develop several differential-geometric and Lie-algebraic tools that allow us to replace the specific Milnor identities of $SU(2)$ by the structural condition $[\fr{m}_i,\fr{m}_j]=\fr{m}_k$, $i,j,k$ pairwise distinct, for the metric eigenspaces $\fr{m}_i$. In contrast to the one-dimensional eigenspaces of $SU(2)$, the subspaces $\fr{m}_i$ may have arbitrary dimension, and the cyclic bracket condition does not provide explicit identities between basis vectors. Our arguments rely only on this structural relation and the compactness of the space, and not on any additional interactions between the basis vectors. As a result, we show that the cyclic Lie-bracket structure essentially controls the volume growth across all radius intervals (Theorem \ref{VolumeGrowth}), thus extending the uniform doubling phenomenon of $SU(2)$ to a much broader class of homogeneous spaces of arbitrarily large dimension. More specifically, consider the following setting.

\begin{assum}\label{Assumption}Let $G/K$ be a homogeneous space with $G$ compact and connected, and let $o:=eK$ be the origin. Fix an $\op{Ad}$-invariant inner product $Q$ on $\fr{g}:=\op{Lie}(G)$, and let $\fr{g}=\fr{k}\oplus \fr{m}$ be a $Q$-orthogonal decomposition, where $\fr{k}=\operatorname{Lie}(K)$ and $\fr{m}$ is canonically identified with $T_o(G/K)$. Assume that  $\fr{m}$ admits a $Q$-orthogonal decomposition   

\begin{equation*} \fr{m}=\fr{m}_1\oplus \fr{m}_2\oplus \fr{m}_3, \end{equation*}
\noindent into $\op{Ad}_K$-invariant subspaces $\fr{m}_i$, such that

\begin{equation}\label{cyclicbracket}[\fr{m}_1,\fr{m}_2]=\fr{m}_3,  \ \ [\fr{m}_1,\fr{m}_3]=\fr{m}_2 \ \ \makebox{and}\ \ [\fr{m}_2,\fr{m}_3]=\fr{m}_1. \end{equation}

\noindent We endow $G/K$ with the $3$-parameter family of $G$-invariant metrics $\alpha_I$, induced by the corresponding $\op{Ad}_K$-invariant inner products   

\begin{equation}\label{MetricParameterForm}\langle \ , \ \ \rangle=\left.\alpha_1^2Q\right|_{\fr{m}_1\times \fr{m}_1}+\left.\alpha_2^2Q\right|_{\fr{m}_2\times \fr{m}_2}+\left.\alpha_3^2Q\right|_{\fr{m}_3\times \fr{m}_3}, \ \ \alpha_1,\alpha_2,\alpha_3>0.\end{equation}

\end{assum}

  Our main result is the following.

\begin{theorem}\label{maintheorem}Under Assumption \ref{Assumption}, there exists a constant $D$, depending only on the space $G/K$, such that $D_{\alpha_I}\leq D$ for all metrics $\alpha_I$ of the form \eqref{MetricParameterForm}.  In particular, the family $(G/K,\alpha_I)$ is uniformly doubling.\end{theorem}

To describe the volume growth of the spaces $(G/K,\alpha_I)$, set 

\[ d_i:=\dim{\fr{m}_i},\ \  i=1,2,3, \]

\noindent and 

\[ n:=d_1+d_2+d_3=\dim{G/K}. \]

\noindent Relabeling the subspaces $\fr{m}_i$, we may assume that

\[ \alpha_1\leq \alpha_2\leq \alpha_3.\]

\noindent Note that since the metrics $\alpha_I$ are $G$-invariant, the volume of $B_{\alpha_I}(p,r)$ is independent of $p\in G/K$, and thus it suffices to consider metric balls centered at the origin $o$. Denote the volume of a ball centered at $o$ with radius $r$ by $Vol_{\alpha_I}(r):=\mu_{\alpha_I}(B_{\alpha_I}(o,r))$. Moreover, define the continuous function $\widetilde{Vol}_{\alpha_I}(r):(0,+\infty)\rightarrow (0,+\infty)$ by

\[ \widetilde{Vol}_{\alpha_I}(r):=\begin{cases} r^n,  & 0<r\leq \frac{\alpha_1\alpha_2}{\alpha_3},\\ 
\big(\frac{\alpha_3}{\alpha_1\alpha_2}\big)^{d_3}r^{n+d_3},  & \frac{\alpha_1\alpha_2}{\alpha_3}<r\leq \alpha_1, \\  
\alpha_1^{d_1}(\frac{\alpha_3}{\alpha_2})^{d_3}r^{n-d_1}, &  \alpha_1<r\leq \alpha_2, \\ 
\alpha_1^{d_1}\alpha_2^{d_2}\alpha_3^{d_3}, & r>\alpha_2  \end{cases}. \]

\noindent We have the following volume growth estimate. 

\begin{theorem}\label{VolumeGrowth} Under Assumption \ref{Assumption}, there exist positive constants $b_1,b_2$, depending only on the space $G/K$, such that 

    \[ b_1\leq \frac{Vol_{\alpha_I}(r)}{\widetilde{Vol}_{\alpha_I}(r)}\leq b_2, \]

    \noindent for every $r>0$ and every metric $\alpha_I$ of the form \eqref{MetricParameterForm}.

\end{theorem}

\subsection{Applications}  To the author's knowledge, there exist only few known classes of uniformly doubling spaces: abelian Lie groups, $SU(2)$ (\cite{EGS18}), quotients of $SU(2)\times \mathbb R^n$ (\cite{EGS25}) and very recently, step-two Carnot groups (\cite{BiZha26}).  Theorem \ref{maintheorem} provides a new broad class of such spaces, and in several cases, establishes uniform doubling for the complete family of their $G$-invariant metrics.

 More specifically, Theorem \ref{maintheorem} immediately recovers Theorem \ref{EGSTheorem}, for $Q$ being the negative of the Killing form and $\fr{m}_i=<e_i>$, $i=1,2,3$. Our results also apply to the class of \emph{generalized Wallach spaces}, of which $SU(2)$ is the simplest example. Generalized Wallach spaces have been classified independently in \cite{Nik16} (correction in \cite{Nik21}) and in \cite{CheKaLi16} for $G$ simple. Their isotropy representation decomposes as $\fr{m}=\fr{m}_1\oplus \fr{m}_2\oplus \fr{m}_3$, where the submodules $\fr{m}_i$ are $\op{Ad}_K$-irreducible and satisfy $[\fr{m}_i,\fr{m}_i]\subseteq \fr{k}$. They include the spaces listed in Table \ref{table}, for which the submodules $\fr{m}_i$ are pairwise inequivalent and hence the metrics $\alpha_I$ exhaust their $G$-invariant metrics. We prove that any generalized Wallach space is uniformly doubling on the set of metrics $\alpha_I$.
 
 Generalized Wallach spaces with $G$ simple are themselves a subclass of compact \emph{$\mathbb Z_2\times \mathbb Z_2$-symmetric spaces}, classified for $G$ compact and simple in \cite{BaGo08} and \cite{Kol09}. The latter are defined as almost effective homogeneous spaces $G/K$ with $G$ compact and connected, for which there exists an injective homomorphism $\rho:\mathbb Z_2\times \mathbb Z_2\rightarrow \op{Aut}(G)$ such that 
 
 \[ G^{\rho}_0\subseteq K \subseteq G^{\rho}.\]

 \noindent Here $G^{\rho}$ denotes the subgroup of fixed points of $\rho(\mathbb Z_2\times \mathbb Z_2)$ and $G^{\rho}_0$ denotes its connected component at $e$.  We prove that compact $\mathbb Z_2\times\mathbb Z_2$-symmetric spaces with $G$ simple satisfy Assumption \ref{Assumption}, as a result of the induced $\mathbb Z_2\times\mathbb Z_2$-grading and the simplicity of $\fr{g}$. More generally, we show the following.
    
\begin{theorem}\label{TheoremWallach}Let $G/K$ be a $\mathbb Z_2\times \mathbb Z_2$-symmetric space with $G$ compact and simple or a generalized Wallach space. Then $G/K$ is uniformly doubling on the family of $G$-invariant metrics $\alpha_I$. Moreover, suppose that $G/K$ is one of the generalized Wallach spaces in Table \ref{table}, the Lie group $SU(2)$ or the Stiefel manifold $V_2\mathbb R^n=SO(n)/SO(n-2)$, $n\geq 3$, $n\neq 4$. Then the metrics $\alpha_I$ exhaust the $G$-invariant metrics of $G/K$ up to isometry, and thus $G/K$ is uniformly doubling on the complete family of $G$-invariant metrics. \end{theorem}

\subsection{A global Poincar\'{e} inequality for compact homogeneous spaces}

An important consequence of the uniform doubling property on Lie groups is the existence of Poincar\'{e} inequalities that hold uniformly for all left-invariant metrics (\cite{EGS18}). For compact homogeneous spaces $(G/K,g)$ with connected $K$, we derive a Poincar\'{e} inequality that relies on the doubling constant of the space. 

\begin{theorem}\label{PoincareTheorem}   Let $G/K$ be a homogeneous space with origin $o=eK$, where $G$ is a compact connected Lie group and $K$ is a closed connected subgroup of $G$. Let $g$ be a $G$-invariant metric on $G/K$, let $\mu_g$ be the corresponding Riemannian volume measure and let $D_g$ be the corresponding doubling constant.  Then the following inequality holds

\begin{equation}\label{PoincareInequality}
 \int_{B_g(p,r)}{|f-f_{B_g(p,r)}|^2d\mu_g}\leq 2D_gr^2\int_{B_g(p,2r)}{\|\nabla_g{f}\|_g^2d\mu_g}, \end{equation}
  
\noindent for all $p\in G/K$, $r>0$ and $f\in \mathcal{C}^{\infty}(G/K)$, where 

\[ f_{B_g(p,r)}=\frac{1}{\mu_g(B_g(p,r))}\int_{B_g(p,r)}{fd\mu_g},\]

\noindent denotes the mean of $f$ over $B_g(p,r)$ and $\nabla_g{f}$ denotes the gradient of $f$ with respect to $g$.  
\end{theorem}

As an immediate corollary, we obtain the following.

\begin{corol}\label{PoincareCorol}Assume that $G$ is compact, connected and $K$ is a connected closed subgroup of $G$. If the space $G/K$ is uniformly doubling on some set of $G$-invariant metrics, then there exists a positive constant $D$ such that the Poincar\'{e} inequality on $G/K$,

\begin{equation}\label{PoincareInequality1}
 \int_{B_g(p,r)}{|f-f_{B_g(p,r)}|^2d\mu_g}\leq 2Dr^2\int_{B_g(p,2r)}{\|\nabla_g{f}\|_g^2d\mu_g}, \end{equation}

\noindent holds uniformly for those metrics. In particular, it holds uniformly for all metrics of the form \eqref{MetricParameterForm} on generalized Wallach spaces and $\mathbb Z_2\times \mathbb Z_2$-symmetric spaces $G/K$ with $G$ compact simple and $K$ connected. It also holds uniformly for all $G$-invariant metrics on the Stiefel manifolds $SO(n)/SO(n-2)$, $n\geq 3$, $n\neq 4$, the Lie group $SU(2)$, and any of the spaces in Table \ref{table}.\end{corol}

\subsection{Methodology and organization of the paper}
In Section \ref{sectionPreliminaries}, we state some preliminary facts about invariant metrics on homogeneous spaces and local coordinate functions. We also set our notation for the rest of the paper, and we prove two structural results that will be useful throughout the paper, lemmas \ref{generator} and \ref{LemmaRelaxed}.   

In Section \ref{SectionRicciLower}, we obtain a lower bound for the Ricci tensor that holds for all diagonal metrics with respect to a fixed decomposition on compact homogeneous spaces (Proposition \ref{RicciBoundProp}).

 After relabeling the subspaces $\fr{m}_i$ so that $\alpha_1\leq \alpha_2\leq \alpha_3$, the volume growth bounds in Theorem \ref{VolumeGrowth} are obtained by dividing the radius of the ball $B_g(o,r)$ into the three distinct intervals, $\Delta_1:=\big(0,\frac{\alpha_1\alpha_2}{\alpha_3}\big]$, $\Delta_2:=\big[\frac{\alpha_1\alpha_2}{\alpha_3}, \alpha_1\big]$ and $\Delta_3:=\big[\alpha_1,\alpha_2\big]$. 

In Section \ref{SectionEuclidean}, we treat the volume growth on $\Delta_1$. On this interval, the volume growth is Euclidean, i.e., comparable to $r^n$, and we prove this using Bishop-Gromov comparison. The boundary $r=\frac{\alpha_1\alpha_2}{\alpha_3}$ arises from the Ricci bound of Section \ref{SectionRicciLower}.  

In Section \ref{SectionSubRiemannian}, we treat the volume growth on $\Delta_2$. On this interval, after the threshold $r=\frac{\alpha_1\alpha_2}{\alpha_3}$, we have $\frac{r}{\alpha_1}\frac{r}{\alpha_2}\geq \frac{r}{\alpha_3}$. Intuitively, movement via directions in $[\fr{m}_1,\fr{m}_2]$ becomes more effective than movement via directions in $\fr{m}_3$. Thus the growth of a ball is comparable to the volume growth in the corresponding sub-Riemannian structure on $G/K$ with horizontal distribution $\fr{m}_1\oplus \fr{m}_2$. Since $[\fr{m}_1,\fr{m}_2]=\fr{m}_3$, the homogeneous dimension of this structure is $d_1+d_2+2d_3$. In turn, the growth is comparable to $r^{d_1+d_2+2d_3}=r^{n+d_3}$. We prove the corresponding lower bound of the volume growth by constructing a suitable $\mathcal{C}^1$ coordinate function via projection of group commutator elements $\exp{(-sX)}\exp{(-tY)}\exp{(sX)}\exp{(tY)}$, generated by directions  $X\in \fr{m}_1$ and $Y\in \fr{m}_2$ (Section \ref{HeisenbergLowerSection}). In fact, we generalize a similar construction in \cite{EGS18}. For the upper bound, we consider a local lift $\gamma$ of a minimizing curve to the group $G$, and then use Chen–Strichartz expansion for the logarithmic path-signature (\cite{Che57}, \cite{Str87}) to estimate the endpoint $\gamma(1)$. Unlike the case for Lie groups, a difficulty in the homogeneous case is that the left Darboux derivative $\gamma^{-1}\dot{\gamma}$ of $\gamma$ has an uncontrolled component in the Lie algebra $\fr{k}$ of $K$. To bypass this difficulty, we prove an auxiliary estimate which controls this component by the $\fr{m}$-component of $\gamma^{-1}\dot{\gamma}$ (Lemma \ref{KcomponentBound}). 

In Section \ref{PostSubRiemannian}, we treat the volume growth on $\Delta_3$. On this interval, after the threshold $r=\alpha_1$, the directions in $\fr{m}_1$ no longer contribute to the volume growth.  Thus the latter is comparable to the growth of the base space $G/H$, through a suitable Riemannian submersion $G/K\rightarrow G/H$ whose vertical space is $\fr{m}_1$. In turn, the volume growth is comparable to $r^{\dim{G/H}}=r^{n-d_1}$. More generally, we relate the volume growth of homogeneous spaces via Riemannian submersions (Proposition \ref{FibrationProp}). 

In Section \ref{MainProofsSection}, we combine the volume bounds for the three intervals to prove the main results, theorems \ref{VolumeGrowth} and \ref{maintheorem}.

In Section \ref{Z2WallachApplications}, we discuss the classes of $\mathbb Z_2\times \mathbb Z_2$-symmetric spaces and generalized Wallach spaces, and we prove Theorem \ref{TheoremWallach}. 

Finally, in Section \ref{PoincareProofSection}, we prove Theorem \ref{PoincareTheorem}. Note that there exist  Poincar\'{e} inequalities for homogeneous spaces $G/K$ (\cite{Ma98}, \cite{SC02}), but they rely on doubling constants for left-invariant metrics on the group $G$, which our results do not provide. To recover the doubling constant of $G/K$, for each $G$-invariant metric $g$ on $G/K$, we construct a family of left-invariant metrics $\bar{g}_{\epsilon}$ on $G$ such that the projection $\pi:(G,\bar{g}_{\epsilon})\rightarrow(G/K,g)$ is a Riemannian submersion and the fibers collapse as $\epsilon \rightarrow 0$. Our result then follows by taking the corresponding limit in the Poincar\'{e} inequality for $G$.

\begin{table}[ht]
\centering
\small
\renewcommand{\arraystretch}{1.18}
\setlength{\tabcolsep}{8pt}

\begin{tabular}{|p{10.2cm}|p{3.5cm}|}
\hline
\centering $G/K$ & \centering\arraybackslash Parameters \\ 
\hline\hline

$M_1\times M_2\times M_3$ &
$M_i$ irreducible compact symmetric spaces \\
\hline

$SO(k+l+m)/(SO(k)\times SO(l)\times SO(m))$ &
$k,l,m\geq 1$; $(k,l,m)\neq(q,1,1)$ up to permutation \\
\hline

$SU(k+l+m)/S(U(k)\times U(l)\times U(m))$ &
$k,l,m\geq 1$ \\

$Sp(k+l+m)/(Sp(k)\times Sp(l)\times Sp(m))$ &
$k,l,m\geq 1$ \\
\hline

$SU(2l)/U(l)$ &
$l\geq 2$ \\

$SO(2l)/(U(1)\times U(l-1))$ &
$l\geq 4$ \\
\hline

$E_6/(SU(4)\times Sp(1)^2\times U(1))$ & \\
$E_6/(Spin(8)\times U(1)^2)$ & \\
$E_6/(Sp(3)\times Sp(1))$ & \\
\hline

$E_7/(Spin(8)\times Sp(1)^3)$ & \\
$E_7/(SU(6)\times Sp(1)\times U(1))$ & \\
$E_7/Spin(8)$ & \\
\hline

$E_8/(Spin(12)\times Sp(1)^2)$ & \\
$E_8/(Spin(8)\times Spin(8))$ & \\
\hline

$F_4/(Spin(5)\times Sp(1)^2)$ & \\
$F_4/Spin(8)$ & \\
\hline
\end{tabular}

\caption{Generalized Wallach spaces with pairwise inequivalent
isotropy summands.}
\label{table}
\end{table}

\section{Preliminaries and notation}\label{sectionPreliminaries}

\subsection{Homogeneous spaces and $G$-invariant metrics}\label{Prel1}

Let $G$ be a compact, connected Lie group and denote by $\fr{g}$ its Lie algebra. For subspaces $V,W$ of $\fr{g}$, we denote by $[V,W]$ the linear span of vectors $[v,w]$ with $v\in V, w\in W$. Since $G$ is compact, we may choose an $\op{Ad}$-invariant inner product $Q$ on $\fr{g}$. We fix this product throughout the rest of the paper. Denote by $\|\cdot \|$ the norm on $\fr{g}$ induced by $Q$, and let 

\[ C_{\fr{g}}:=\|[ \cdot, \cdot ]\|_{op}=\sup_{\|X\|=\|Y\|=1}\|[X,Y]\|,\]

\noindent be the operator norm of the Lie bracket on $\fr{g}$, so that

\begin{equation}\label{Cg} \| [X,Y]\|\leq C_{\fr{g}}\|X\| \|Y\| \ \ \makebox{for all} \ \ X,Y\in \fr{g}. \end{equation}   

Let $K$ be a closed subgroup of $G$ and consider the homogeneous space $G/K$ with origin $o:=eK$. Denote by $\fr{k}$ the Lie algebra of $K$, and consider the $Q$-orthogonal \emph{reductive decomposition}

\[ \fr{g}=\fr{k}\oplus \fr{m}. \]

\noindent Then $\fr{m}$ is $\op{Ad}_K$-invariant, and we can canonically identify $\fr{m}$ with $T_o(G/K)$ via the differential $d\pi_e$ of the canonical projection $\pi:G\rightarrow G/K$.  For $x\in G$, let $L_x,R_x:G\rightarrow G$ denote the left and right translations in $G$ respectively, and let $\tau_x:G/K\rightarrow G/K$, $yK\mapsto (xy)K$ denote the left translation in $G/K$, so that $\tau_x\circ \pi=\pi\circ L_x$. A Riemannian metric $g$ on $G/K$ is called \emph{$G$-invariant} if the left translations $\tau_x$, $x\in G$, are isometries of $(G/K,g)$. Then $(G/K,g)$ is called a \emph{Riemannian homogeneous space}.

 The $G$-invariant Riemannian metrics on $G/K$ are in bijection with $\op{Ad}_K$-invariant inner products on $\fr{m}$. In the sequel, we will make no distinction between a $G$-invariant metric and the corresponding inner product on $\fr{m}$.  

\subsection{Invariant measures and volume functions}\label{Prel2}
 
For a Riemannian manifold $(M,g)$, denote by $B_g(p,r)$ the open ball of radius $r$ centered at $p\in M$, with respect to the distance $d_g$ induced from the metric $g$. Let $\mu_g$ denote the corresponding Riemannian volume measure.
 
 We turn to homogeneous spaces $G/K$. Note that for any $G$-invariant metric $g$ on $G/K$, the left-translation invariance and the triangle inequality imply

 \begin{equation}\label{triangle} d_g(o, \pi(x_1x_2\cdots x_l))\leq \sum_{i=1}^l{d_g(o,\pi(x_i))}, \end{equation}
 
\noindent for all $x_1,\dots,x_l\in G$.  We set 
  \[  Vol_g(r):=\mu_g(B_g(o,r)), \] 

\noindent  For simplicity, we denote by $\mu_0$ the $G$-invariant Riemannian volume measure on $G/K$ induced from the normal homogeneous metric  $\left.Q\right|_{\fr{m}\times \fr{m}}$, which will be our reference metric. For the rest of the paper, unless otherwise stated, assume that $G/K$ satisfies Assumption \ref{Assumption}, and thus is endowed with the $G$-invariant metrics $\alpha_I$ given by relation \eqref{MetricParameterForm}, with $\alpha_1\leq \alpha_2\leq \alpha_3$, up to a possible relabeling of the submodules $\fr{m}_i$. If the $\op{Ad}_K$-submodules $\fr{m}_i$ are irreducible and pairwise inequivalent, Schur's lemma implies that the metrics $\alpha_I$ exhaust the $G$-invariant metrics on $G/K$. We set 
\[ \Lambda_{\alpha_I}:=\alpha_1^{d_1}\alpha_2^{d_2}\alpha_3^{d_3},\]
 
\noindent so that

\begin{equation}\label{measure}\mu_{\alpha_I}=\Lambda_{\alpha_I} \mu_0.\end{equation}
  
\noindent Moreover, for a $G$-invariant metric $g$ we set  

\[ V_g(r):=\mu_0(B_g(o,r)),\]
 
 \noindent so that Equation \eqref{measure} yields

\begin{equation*}Vol_{\alpha_I}(r)=\Lambda_{\alpha_I}V_{\alpha_I}(r).\end{equation*}

\subsection{The Jacobian of a coordinate map}\label{Jacobian} Denote by $m$ the Lebesgue measure on $\fr{m}$, induced by the inner product $\left.Q\right|_{\fr{m}\times \fr{m}}$.  Let $U$ be an open set of $\fr{m}$ and let $F:U\rightarrow G/K$ be a diffeomorphism onto its image. Following \cite{EGS18}, we define the Jacobian of $F$ as the continuous function $J_F:U\rightarrow (0,+\infty)$ such that

\[ \mu_0(F(S))=\int_S{J_F dm},\]
 
\noindent for any measurable $S\subseteq U$.

\subsection{Local coordinate functions on $G/K$}\label{sectioncoordinate}
 Consider a $n$-dimensional homogeneous space $G/K$ with the $Q$-orthogonal reductive decomposition

\begin{equation}\label{reductive123} \fr{g}=\fr{k}\oplus\fr{m}_1\oplus \cdots \oplus{\fr{m}_s}.\end{equation}

\noindent We set 

\[ d_i:=\dim{\fr{m}_i}, \ \ i=1,\dots,s, \]

\noindent so that $n=d_1+\cdots +d_s$.  Fix a $Q$-orthonormal basis of $\fr{m}$,

\[ \{E_i^j: i=1,\dots,s, \ j=1,\dots d_i\}, \]

\noindent  adapted to the decomposition \eqref{reductive123}. This basis will remain fixed throughout the rest of the paper. For the spaces satisfying Assumption \ref{Assumption}, it will be adapted to the $Q$-orthogonal decomposition $\fr{m}=\fr{m}_1\oplus \fr{m}_2\oplus \fr{m}_3$.  We identify $\fr{m}$ with $\mathbb R^n$ when it is convenient. We will henceforth use the notation $X=(x_i^{j})_{j=1}^{d_i}$ for any vector $X=\sum_{j=1}^{d_i}{x_i^{j}E_i^{j}}\in \fr{m}_i$, $i=1,\dots,s$. 
 
 We consider two kinds of local coordinate functions on $G/K$. The exponential coordinates of mixed kind, defined by $\Phi:\fr{m}\rightarrow G/K$, with 

\[ \Phi(X_1,\dots,X_s)=\pi\big(\exp{X_1}\cdots \exp{X_s}\big), \ \ X_i\in \fr{m}_i,\ \ i=1,\dots,s,\]

\noindent and the standard exponential coordinates $\Psi:\fr{m}\rightarrow G/K$, with

\[ \Psi(X)=\pi(\exp{X}), \  \  X\in \fr{m}.\]

\noindent Both functions $\Phi$ and $\Psi$ define local coordinates around $o=\Phi(0)=\Psi(0)$, since $d\Phi_0=d\Psi_0=\op{Id}_{\fr{m}}$, under the identification $\fr{m}=T_o(G/K)$ via $d\pi_e$. By the inverse function theorem, there exists a neighborhood $U$ of $0$ in $\fr{m}$ such that both maps $\left.\Phi\right|_U:U\rightarrow \Phi(U)$ and $\left.\Psi\right|_U:U\rightarrow \Psi(U)$ are diffeomorphisms.  Hence by choosing $\epsilon_0>0$ such that $[-\epsilon_0,\epsilon_0]^n\subset U$, the continuity of the Jacobians $J_{\Phi}$ and $J_{\Psi}$ yields the following.  

  \begin{lemma}\label{boundedcoordinates}There exist positive constants $\epsilon_0$ and $C_0$, depending only on the space $G/K$ and the fixed product $Q$, such that $\mu_0(\Phi(S))\geq C_0m(S)$ and $\mu_0(\Psi(S))\geq C_0m(S)$ for any measurable $S\subseteq [-\epsilon_0,\epsilon_0]^n$, where $m$ is the Lebesgue measure on $\fr{m}$.  \end{lemma}
 
\noindent An advantage of the exponential coordinates of any kind is the fact that

\begin{equation}\label{triangle2}d_g(o,\pi(\exp{X}))\leq \sqrt{g(d\pi_e(X),d\pi_e(X))},\ \ \makebox{for all}\ \ X\in \fr{g}. \end{equation}

\subsection{Two structural lemmas for homogeneous spaces}\label{structural} We close the preliminary section with two structural facts for homogeneous spaces, concerning mostly the spaces $G/K$ of Assumption \ref{Assumption}.

\begin{lemma}\label{generator}Let $G/K$ be a homogeneous space satisfying Assumption \ref{Assumption}.  Set $N:=d_1d_2$ and consider the $Q$-orthonormal basis $\{E_i^j: i=1,2,3, \ j=1,\dots , d_i\}$, adapted to the decomposition $\fr{m}=\fr{m}_1\oplus \fr{m}_2\oplus \fr{m}_3$. There exist vectors $A^k_l\in \fr{m}_1$, $B^k_l\in \fr{m}_2$, $k=1,\dots,d_3$, $l=1,\dots, N$, such that 
 
 \begin{equation}\label{generatoreq}E_3^k=\sum_{l=1}^N{[A^k_l,B^k_l]}, \ \ k=1,\dots, d_3.  \end{equation}
 \end{lemma}

  \begin{proof}  Since $\fr{m}_3=[\fr{m}_1,\fr{m}_2]$, we can write
  
  \[ E_3^k=\sum_{i=1}^{d_1}{\sum_{j=1}^{d_2}{c_{ij}^k[E_1^i,E_2^j]}}, \ \ k=1,\dots, d_3, \] 

  \noindent for suitable constants $c^k_{ij}$. In other words, each $E_3^k$, $k=1,\dots,d_3$, is written as the sum of $N$ Lie brackets, which yields relation \eqref{generatoreq}.  \end{proof}

We will also use the following lemma.

\begin{lemma}\label{LemmaRelaxed} Let $G/K$ be a homogeneous space, admitting an orthogonal reductive decomposition

   \begin{equation*} \fr{g}=\fr{k}\oplus \fr{m}_1\oplus \fr{m}_2\oplus \fr{m}_3, \end{equation*}

\noindent with respect to an $\op{Ad}$-invariant inner product $Q$ on $\fr{g}$, such that the spaces $\fr{m}_i$ are $\op{Ad}_K$-invariant and 

   \begin{equation}\label{RELCOND}[\fr{m}_i,\fr{m}_j]\subseteq\fr{m}_k,  \ \ \makebox{$i,j,k$ pairwise distinct}. \end{equation}

\noindent Then the following are true:\\

\noindent 1. If $[\fr{m}_i,\fr{m}_j]=\fr{m}_k$, for $i,j,k$ pairwise distinct, then for each $i=1,2,3$, it holds

\[ [\fr{m}_i,\fr{m}_i]\subseteq \fr{k}, \]

and thus the space $\fr{h}_i:=\fr{k}\oplus \fr{m}_i$ is a Lie subalgebra of $\fr{g}$.\\

\noindent 2. The group 

\[ N_G(\fr{h}_i)=\{g\in G:\op{Ad}_g\fr{h}_i\subseteq \fr{h}_i\} \]

\noindent is a closed subgroup of $G$ that contains $K$. Moreover, if $[\fr{m}_i,\fr{m}_j]=\fr{m}_k$, for $i,j,k$ pairwise distinct, then the Lie algebra of $N_G(\fr{h}_i)$ is $\fr{h}_i$.\\

\noindent 3. If the $\op{Ad}_K$-submodules $\fr{m}_i$ are irreducible, then either 

\begin{eqnarray*}[\fr{m}_i,\fr{m}_j]&=&\{0\}, \ \ \makebox{for all $i\neq j$, or} \\ \\  
{}[\fr{m}_i,\fr{m}_j]&=&\fr{m}_k, \ \ \makebox{for all $i,j,k$ pairwise distinct}. \end{eqnarray*}
 
   \end{lemma}

\begin{proof} Choose arbitrary $i=1,2,3$. We have $[\fr{k},\fr{k}]\subseteq \fr{k}$ and $[\fr{k},\fr{m_i}]\subseteq \fr{m}_i$ by the $\op{Ad}_K$-invariance of $\fr{m}_i$. Moreover, for $i,j,k$ pairwise distinct, the $\op{Ad}$-invariance of $Q$, along with the Lie brackets \eqref{RELCOND} of the spaces $\fr{m}_i$, yields 

\[ Q([\fr{m}_i,\fr{m}_i],\fr{m}_j\oplus \fr{m}_k)\subseteq Q(\fr{m}_i,[\fr{m}_i,\fr{m}_j\oplus \fr{m}_k])\subseteq Q(\fr{m}_i, \fr{m}_k\oplus \fr{m}_j)=\{0\}, \] 

\noindent and thus $[\fr{m}_i,\fr{m}_i]\subseteq \fr{k}\oplus \fr{m}_i$. It remains to show that $[\fr{m}_i,\fr{m}_i]$ is $Q$-orthogonal to $\fr{m}_i$. We will prove it for $i=3$, as the other cases are similar. To this end, since $\fr{m}_3=[\fr{m}_1,\fr{m}_2]$, we may write any $Z\in \fr{m}_3$ as 

\[ Z=\sum_{l=1}^{d_1d_2}{[A_l,B_l]},\]
 
 \noindent where $A_l\in \fr{m}_1$ and $B_l\in \fr{m}_2$ (c.f. Lemma \ref{generator}).  Then for any $X,Y\in \fr{m}_3$, the $\op{Ad}$-invariance of $Q$ and the Jacobi identity yield

 \begin{eqnarray*}Q([X,Y],Z)&=&\sum_{l=1}^{d_1d_2}{Q\big([X,Y],[A_l,B_l]\big)}=\sum_{l=1}^{d_1d_2}{Q\big(\big[[X,Y],A_l\big],B_l\big)}\\ \\ 
 &=&\sum_{l=1}^{d_1d_2}{\bigg(Q\big(\big[X,[Y,A_l]\big],B_l\big)-Q\big(\big[Y,[X,A_l]\big],B_l\big)\bigg)}.\end{eqnarray*}
 
 \noindent In view of relations \eqref{RELCOND}, the terms $\big[X,[Y,A_l]\big]$ and $\big[Y,[X,A_l]\big]$ lie in $\fr{m}_1$. Since $B_l\in \fr{m}_2$ and the submodules $\fr{m}_l$ are $Q$-orthogonal, the right hand side of the above equation vanishes, and thus $Q([X,Y],Z)=0$. Therefore, $[\fr{m}_3,\fr{m}_3]$ is $Q$-orthogonal to $\fr{m}_3$, which yields $[\fr{m}_3,\fr{m}_3]\subseteq \fr{k}$. The other two relations follow in an identical manner, concluding the proof of part 1.

We will prove part 2. for $i=1$, as the other cases are similar. The continuity of $\op{Ad}$ implies that $N_G(\fr{h}_1)$ is closed. Moreover, by the $\op{Ad}_K$-invariance of $\fr{m}_1$, we have $\op{Ad}_K\fr{h}_1\subseteq \fr{h}_1$, and thus $K\subseteq N_G(\fr{h}_1)$. Now the Lie algebra of $N_G(\fr{h}_1)$ is the normalizer

\[ \fr{n}_{\fr{g}}(\fr{h}_1)=\{X\in \fr{g}:[X,\fr{h}_1]\subseteq \fr{h}_1\}\] 

\noindent (see e.g. \cite{Ant12}). It is immediate that $\fr{h}_1\subseteq \fr{n}_{\fr{g}}(\fr{h}_1)$. Conversely, let $X\in \fr{n}_{\fr{g}}(\fr{h}_1)$. Then 

\[ X=X_{\fr{h}_1}+X_2+X_3, \ \ \makebox{with} \ \ X_{\fr{h}_1}\in \fr{h}_1,\ \ X_2\in \fr{m}_2,\ \  X_3\in \fr{m}_3.\]
 
 \noindent We will show that $X_2=X_3=0$, which will conclude the proof that $\fr{n}_{\fr{g}}(\fr{h}_1)=\fr{h}_1$, and in turn the proof of part 2. Indeed, since $[X,\fr{h}_1]\subseteq \fr{h}_1$, we have $[X,\fr{m}_1]\subseteq \fr{h}_1$, i.e., 

\[ [X_2,\fr{m}_1]+[X_3,\fr{m}_1]\subseteq \fr{h}_1.\]
  
  \noindent On the other hand, relations \eqref{RELCOND} imply that the left hand side of the above relation lies in $\fr{m}_3\oplus \fr{m}_2$, and thus $[X_2,\fr{m}_1]=[X_3,\fr{m}_1]=0$. Since $\fr{m}_2=[\fr{m}_1,\fr{m}_3]$, we obtain

 \[Q(X_2,\fr{m}_2)\subseteq Q(X_2,[\fr{m}_1,\fr{m}_3])\subseteq Q([X_2,\fr{m}_1],\fr{m}_3)=\{0\}.\]
\noindent Since $X_2\in \fr{m}_2$, the above relation yields $X_2=0$. Similarly, $X_3=0$, concluding the proof of part 2.

For part 3., choose $i,j,k=1,2,3$ pairwise distinct. Both spaces $\fr{m}_k$ and $[\fr{m}_i,\fr{m}_j]\subseteq \fr{m}_k$ are $\op{Ad}_K$-invariant. Hence the irreducibility of $\fr{m}_k$ implies that $[\fr{m}_i,\fr{m}_j]=\{0\}$ or $[\fr{m}_i,\fr{m}_j]=\fr{m}_k$. If the first case is true, say without any loss of generality that $[\fr{m}_1,\fr{m}_2]=\{0\}$, then 

\[ Q([\fr{m}_2,\fr{m}_3],\fr{m}_1)\subseteq Q(\fr{m}_3,[\fr{m}_1,\fr{m}_2])=\{0\}. \]

\noindent Therefore, $[\fr{m}_2,\fr{m}_3]$ is $Q$-orthogonal to $\fr{m}_1$ which, given that $[\fr{m}_2,\fr{m}_3]\subseteq \fr{m}_1$ yields $[\fr{m}_2,\fr{m}_3]=\{0\}$. Similarly, $[\fr{m}_1,\fr{m}_3]=\{0\}$. This completes the proof.\end{proof}

\section{A lower bound for the Ricci tensor on diagonal metrics}\label{SectionRicciLower}
In this section, we derive a lower bound for the Ricci tensor on compact homogeneous spaces with respect to diagonal metrics. Let $(G/K,g)$ be a $n$-dimensional Riemannian homogeneous space, $n\geq 2$, with $G$ compact, admitting the reductive decomposition $\fr{g}=\fr{k}\oplus \fr{m}$. Let $B$ denote the Killing form of $\fr{g}$, which is negative semidefinite (it is negative definite if $G$ is semisimple), and let $(E_i)_{i=1}^n$ be a $g$-orthonormal basis on $\fr{m}$.  Since $G$ is compact and hence unimodular, by \cite{Bes}, Corollary 7.38, the Ricci tensor of $g$ is given by

\begin{equation}\label{RicciTensorBesse}\op{Ric}_g(X,X)=-\frac{1}{2}\sum_{i=1}^n{\|[X,E_i]_{\fr{m}}\|_g^2}-\frac{1}{2}B(X,X)+\frac{1}{4}\sum_{i,j=1}^n{\big(g([E_i,E_j]_{\fr{m}},X)\big)^2}, \ \ X\in \fr{m}, \end{equation}

\noindent where the subscript $\fr{m}$ on the Lie bracket denotes the orthogonal projection on $\fr{m}$ with respect to the reductive decomposition, and $\|\cdot \|_g$ denotes the norm on $\fr{m}$ induced from the metric $g$.

We prove the following bound.

\begin{prop}\label{RicciBoundProp}Let $G/K$ be a homogeneous space with $G$ compact, let $Q$ be an $\op{Ad}$-invariant inner product on $\fr{g}$ and consider the $Q$-orthogonal reductive decomposition

\[ \fr{g}=\fr{k}\oplus \fr{m}. \]

\noindent Assume further that $\fr{m}$ admits the $Q$-orthogonal decomposition

\begin{equation}\label{SDecomposition} \fr{m}=\fr{m}_1\oplus \cdots \oplus \fr{m}_s, \end{equation}

\noindent into $\op{Ad}_K$-invariant subspaces $\fr{m}_i$. Endow $G/K$ with the $G$-invariant metric 

\[ g:=\alpha_1^2\left.Q\right|_{\fr{m}_1\times \fr{m}_1}+\cdots + \alpha_s^2\left.Q\right|_{\fr{m}_s\times \fr{m}_s}, \ \ \alpha_i>0. \]

\noindent We set
 \[A_g:=\max\{\frac{\alpha_k}{\alpha_i\alpha_j}: [\fr{m}_i,\fr{m}_j]_{\fr{m}_k}\neq \{0\}\}, \]

\noindent and $A_g:=0$ if $[\fr{m}_i,\fr{m}_j]_{\fr{m}_k}=\{0\}$ for all $i,j,k=1,\dots,s$. Here the subscript $\fr{m}_k$ on the Lie bracket denotes the orthogonal projection on $\fr{m}_k$ with respect to the decomposition \eqref{SDecomposition}.  Then there exists a positive constant $C_{ric}$, depending only on the space $G/K$, the product $Q$ and the decomposition \eqref{SDecomposition}, such that 

\[ \op{Ric}_g\geq -C_{ric}A_g^2g. \]
\end{prop}

\begin{proof} Observe that the decomposition \eqref{SDecomposition} is also $g$-orthogonal. Set $d_i:=\dim{\fr{m}_i}$, $i=1,\dots,s$, and let $\{E_i^{j}\}$, $i=1,\dots ,s$, $j=1,\dots, d_i$, be a $g$-orthonormal basis adapted to the decomposition \eqref{SDecomposition}.  Since the killing form $B$ of $\fr{g}$ is negative semidefinite, formula \eqref{RicciTensorBesse} yields
     \begin{equation}\label{FirstRicciBound} \op{Ric}_g(X,X)\geq -\frac{1}{2}\sum_{i=1}^s{\sum_{j=1}^{d_i}{\|[X,E_i^j]_{\fr{m}}\|_g^2}}. \end{equation}
     
     \noindent Writing $X=\sum_{l=1}^s{X_l}$, $X_l\in \fr{m}_l$, we have

\[ [X,E_i^j]_{\fr{m}}=\sum_{k=1}^s{\sum_{l=1}^s{[X_l,E_i^j]_{\fr{m}_k}}}. \]
 
\noindent By using successively the $g$-orthogonality of $\fr{m}_k$, $k=1,\dots,s$, the triangle inequality, the Cauchy-Schwarz inequality and the definition of $g$, we obtain

\begin{eqnarray*} \|[X,E_i^j]_{\fr{m}}\|_g^2&=&\sum_{k=1}^s{\|\sum_{l=1}^s{[X_l,E_i^j]_{\fr{m}_k}}}\|_g^2\leq \sum_{k=1}^s{\big(\sum_{l=1}^s{\|[X_l,E_i^j]_{\fr{m}_k}}\|_g\big)^2}\\ \\ 
&\leq& \sum_{k=1}^s{\big({s\sum_{l=1}^s{\|[X_l,E_i^j]_{\fr{m}_k}\|_g^2}}\big)}=s\sum_{k=1}^s{\sum_{l=1}^s{\alpha_k^2\|[X_l,E_i^j]_{\fr{m}_k}\|^2}}.\end{eqnarray*}

Using the definition of $A_g$ in the hypothesis, the facts that $\|X_l\|_g=\alpha_l\|X_l\|$ and $\|E_i^j\|=\frac{1}{\alpha_i}$, as well as relation \eqref{Cg}, each non-zero term in the above sum satisfies
 
 \begin{equation*}\alpha_k^2\|[X_l,E_i^j]_{\fr{m}_k}\|^2\leq C_{\fr{g}}^2\alpha_k^2\|X_l\|^2\|E_i^j\|^2=C_{\fr{g}}^2(\frac{\alpha_k}{\alpha_l\alpha_i})^2\|X_l\|^2_g\leq C_{\fr{g}}^2A_g^2\|X_l\|^2_g.\end{equation*}

\noindent Therefore, 

\[  \|[X,E_i^j]_{\fr{m}}\|_g^2\leq (C_{\fr{g}}A_g)^2s\sum_{k=1}^s{\sum_{l=1}^s{\|X_l\|^2_g}}\leq (sC_{\fr{g}}A_g)^2\|X\|_g^2=(sC_{\fr{g}}A_g)^2g(X,X).   \]

\noindent Along with relation \eqref{FirstRicciBound}, and the fact that $\sum_{i=1}^s{d_i}=n$, we obtain

\[ \op{Ric}_g(X,X)\geq -\frac{1}{2}\sum_{i=1}^s{\sum_{j=1}^{d_i}{(sC_{\fr{g}}A_g)^2g(X,X)}}= -\frac{n}{2}(sC_{\fr{g}}A_g)^2g(X,X), \]
  
\noindent which yields the desired result for $C_{ric}:=\max\{1,\frac{n}{2}(sC_{\fr{g}})^2\}$.    \end{proof}

For our case, we immediately obtain the following.

\begin{corol}\label{RicciBoundCorol} Let $(G/K,\alpha_I)$ be a homogeneous space satisfying Assumption \ref{Assumption}, where $\alpha_1\leq\alpha_2\leq\alpha_3$. Then there exists a positive constant $C_{ric}$, depending only on the space $G/K$, such that  
\[ \op{Ric}_{\alpha_I}\geq -C_{ric}(\frac{\alpha_3}{\alpha_1\alpha_2})^2\alpha_I. \] \end{corol}

\begin{proof} Assumption \ref{Assumption}, along with part 1. of Lemma \ref{LemmaRelaxed}, respectively yield $[\fr{m}_i,\fr{m}_j]=\fr{m}_k$, for $i,j,k$ pairwise distinct, and $[\fr{m}_i,\fr{m}_i]\subseteq \fr{k}$, for $i=1,2,3$. Hence

\[ [\fr{m}_i,\fr{m}_i]_{\fr{m}}=\{0\}, \ \ i=1,2,3.\]

\noindent Therefore, the only non-zero terms in the definition of $A_{\alpha_I}$ are given by pairwise distinct indices $i,j,k$. Thus, under the condition $\alpha_1\leq\alpha_2\leq\alpha_3$, we have

\[ A_{\alpha_I}=\max\{\frac{\alpha_3}{\alpha_1\alpha_2}, \frac{\alpha_2}{\alpha_1\alpha_3}, \frac{\alpha_1}{\alpha_2\alpha_3}\}=\frac{\alpha_3}{\alpha_1\alpha_2}.\]
 
 \noindent The result then follows from Proposition \ref{RicciBoundProp}.   \end{proof}

\begin{remark}\label{RemarkRicciProduct1} We note that if the submodules $\fr{m}_i$ in the decomposition \eqref{SDecomposition} satisfy $[\fr{m}_i,\fr{m}_j]_{\fr{m}}=\{0\}$, for all $i,j=1,\dots,s$, then $A_g=0$ and hence $\op{Ric}_g\geq 0$.\end{remark}

\section{Euclidean interval of the radius}\label{SectionEuclidean}

The results of this section hold for any compact homogeneous space with a diagonal metric, so we consider the general setting of the previous section. We determine the first radius interval on which the volume growth is uniformly comparable to $r^n$. The lower Ricci bound of the previous section allows us to apply Bishop--Gromov comparison on this interval.
 
 More specifically, let $G/K$ be a compact homogeneous space with the $Q$-orthogonal reductive decomposition 

\[ \fr{g}=\fr{k}\oplus \fr{m}_1\oplus \cdots \oplus \fr{m}_s, \]

\noindent into $\op{Ad}_K$-invariant subspaces $\fr{m}_i$, endowed with the $G$-invariant metric 

\begin{equation}\label{DiagonalEuclidean} g:=\alpha_1^2\left.Q\right|_{\fr{m}_1\times \fr{m}_1}+\cdots + \alpha_s^2\left.Q\right|_{\fr{m}_s\times \fr{m}_s}, \ \ \alpha_i>0. \end{equation}

\noindent As in Section \ref{sectionPreliminaries}, set 

\begin{equation}\label{lambdaGi} \Lambda_g:=\prod_{i=1}^s\alpha_i^{d_i},\end{equation}
 
 \noindent and consider the functions $Vol_g(r)=\mu_g(B_g(o,r))$ and $V_g(r)=\mu_0(B_g(o,r))$, so that 

\[ Vol_g(r)=\Lambda_gV_g(r). \]

\noindent We prove the following. 

\begin{prop}\label{EuclideanBoundGeneral} Let $A_g$ be as in Proposition \ref{RicciBoundProp} and set

\[ \alpha_{min}:=\min\{\alpha_i, i=1,\dots,s\}. \]

\noindent Assume that $A_g>0$. Then there exist positive constants $C_1,C_2$, depending only on the space $G/K$, such that

\[  C_1r^n\leq Vol_g(r)\leq C_2r^n, \ \ 0< r\leq \min\{\alpha_{min}, A_g^{-1}\},\]

\noindent uniformly for all $\alpha_1, \dots,\alpha_s$.
 \end{prop}

\begin{proof} We set 

\[ k_g:=\frac{C_{ric}}{n-1}A_g^2. \]

\noindent By Proposition \ref{RicciBoundProp}, we have $\op{Ric}_g\geq -(n-1)k_gg$, and thus Bishop-Gromov comparison (see e.g. \cite{Peter}, Lemma 35) yields

\[ Vol_g(r)\leq Vol_{-k_g}(r), \ \ r>0,\]

\noindent where $Vol_{-k_g}(r)$ denotes the volume of the ball of radius $r$ in the simply connected $n$-dimensional space form of constant sectional curvature $-k_g$. Now the volume $Vol_{-k_g}(r)$ is given by (\cite{Chav}, Chapter III, Section 4)

\begin{equation*}Vol_{-k_g}(r)=c_{n-1}\int_0^r{\bigg(\frac{\sinh{(\sqrt{k_g}t)}}{\sqrt{k_g}}\bigg)^{n-1}dt}, \end{equation*}

\noindent where $c_{n-1}$ is the $n-1$ dimensional area of the unit sphere in $\mathbb R^n$.
For $x\geq 0$, we have $\sinh{(x)}\leq xe^x$. Therefore, for $0< t\leq r$, 

\[ \bigg(\frac{\sinh(\sqrt{k_g}t)}{\sqrt{k_g}}\bigg)^{n-1}\leq t^{n-1}e^{(n-1)\sqrt{k_g}t} \leq t^{n-1}e^{(n-1)\sqrt{k_g}r},\]
 
 \noindent and integrating gives 

\begin{equation}\label{Sampanis} Vol_g(r)\leq Vol_{-k_g}(r)\leq \frac{c_{n-1}}{n}r^n e^{(n-1)\sqrt{k_g}r}. \end{equation}

\noindent Assume that $0< r\leq \min\{\alpha_{min}, A_g^{-1}\}$. Since $r\leq A_g^{-1}$, a straightforward calculation shows that $(n-1)\sqrt{k_g}r\leq \sqrt{(n-1)C_{ric}}$, and thus Equation \eqref{Sampanis} yields

  \[ Vol_g(r)\leq C_2r^n, \ \ \makebox{with} \ \ C_2:=\frac{c_{n-1}}{n}e^{\sqrt{(n-1)C_{ric}}},\]

 \noindent which yields the upper bound.

For the lower bound, we consider the exponential coordinates of the mixed kind $\Phi:U\rightarrow G/K$ given in Section \ref{sectioncoordinate}. Let $\epsilon_0,C_0$ be as in Lemma \ref{boundedcoordinates}, so that 

 \begin{equation}\label{Jolene}\mu_0(\Phi(S))\geq C_0m(S), \end{equation} 

\noindent for any measurable $S\subseteq [-\epsilon_0,\epsilon_0]^n$. Set 

\[ D_0:=\sum_{i=1}^s{\sqrt{d_i}}\ \ \makebox{and} \ \ D_1:=\max\{2D_0,\epsilon_0^{-1}\}. \]

\noindent Consider the box

\[ S:=\prod_{i=1}^s{\bigg[-\frac{r}{D_1\alpha_i},\frac{r}{D_1\alpha_i}\bigg]^{d_i}}\subset \fr{m},  \ \ \makebox{with} \ \ \bigg[-\frac{r}{D_1\alpha_i},\frac{r}{D_1\alpha_i}\bigg]^{d_i}\subset \fr{m}_i.\]

\noindent Since $0<r\leq \alpha_{min}\leq \alpha_i$, $i=1,\dots, s$,  we have $\frac{r}{D_1\alpha_i}\leq \epsilon_0$ for all $i=1,\dots, s$, and hence 

 \[ S \subseteq [-\epsilon_0,\epsilon_0]^n.\]

\noindent On the other hand, for any $X=X_1+\cdots +X_s\in S$, $X_i\in \fr{m}_i$, the inequalities \eqref{triangle} and \eqref{triangle2}, as well as  the definitions of the metric $g$, the coordinates $\Phi$ and the constants $D_0,D_1$, yield

\begin{eqnarray*}d_g(o,\Phi(X))&\leq &\sum_{i=1}^s{d_g(o,\pi(\exp{X_i}))}\leq\sum_{i=1}^s{\sqrt{g(X_i,X_i)}}\\ \\
&=&\sum_{i=1}^s{\alpha_i\|X_i\|}\leq \sum_{i=1}^s{\alpha_i\sqrt{d_i}\frac{r}{D_1\alpha_i}}\leq \frac{r}{D_1}\sum_{i=1}^s{\sqrt{d_i}}=\frac{D_0}{D_1}r<r.\end{eqnarray*}

\noindent Therefore, $\Phi(S)\subseteq B_g(o,r)$. Along with relation \eqref{Jolene}, the definition of the function $V_g(r)$, and the definition \eqref{lambdaGi} of $\Lambda_g$, we obtain

\[ V_g(r)=\mu_0(B_g(o,r))\geq \mu_0(\Phi(S))\geq C_0m(S)=C_0(\frac{2r}{D_1})^n\Lambda_g^{-1}.\]

\noindent We conclude that 

\[ Vol_g(r)=\Lambda_gV_g(r)\geq C_1r^n, \]
 
 \noindent with $C_1:=C_0(\frac{2}{D_1})^n$, thus completing the proof for the lower bound.\end{proof} 

\begin{remark}\label{UniformDoublingProduct}We note that if the submodules $\fr{m}_i$ satisfy $[\fr{m}_i,\fr{m}_j]_{\fr{m}}=\{0\}$, for all $i,j=1,\dots,s$, then by Remark \ref{RemarkRicciProduct1} we have $Ric_g\geq 0$, and thus Bishop-Gromov comparison yields

\[ \frac{Vol_g(2r)}{Vol_g(r)}\leq 2^n, \]

\noindent i.e., the space $G/K$ is uniformly doubling on the complete set of metrics \eqref{DiagonalEuclidean}. \end{remark}

For the spaces satisfying Assumption \ref{Assumption}, we obtain the following.

\begin{corol}\label{CorolEuclideanFinal} Let $(G/K,\alpha_I)$ be a homogeneous space satisfying Assumption \ref{Assumption}. Then there exist positive constants $C_1,C_2$, depending only on the space $G/K$, such that

\[  C_1r^n\leq Vol_{\alpha_I}(r)\leq C_2r^n, \ \ 0< r\leq \frac{\alpha_1\alpha_2}{\alpha_3},\] 

\noindent uniformly in $\alpha_1\leq \alpha_2\leq \alpha_3$. \end{corol}

\begin{proof} By Corollary \ref{RicciBoundCorol}, we have $A_g=A_{\alpha_I}=\frac{\alpha_3}{\alpha_1\alpha_2}$. Given that $A_{\alpha_I}^{-1}\leq \alpha_1=\alpha_{min}$, the result follows from Proposition \ref{EuclideanBoundGeneral}. \end{proof}

\section{Sub-Riemannian interval of the radius}\label{SectionSubRiemannian}

In this section, we consider the second interval for the radius of a metric ball. The volume growth evolves from Euclidean to sub-Riemannian, as the expensive directions in $\fr{m}_3$ can now be generated more efficiently by directions in $\fr{m}_1$ and $\fr{m}_2$, via the relation $\fr{m}_3=[\fr{m}_1,\fr{m}_2]$. This leads to a volume growth of order $r^{n+d_3}$. Throughout this section, we assume that $G/K$ is a homogeneous space satisfying Assumption \ref{Assumption} with $\alpha_1\leq \alpha_2\leq \alpha_3$.

\subsection{Sub-Riemannian lower bound}\label{HeisenbergLowerSection}
The main result of this subsection is the following lower bound.

\begin{prop}\label{HeisenbergLower} Let $(G/K,\alpha_I)$ be a homogeneous space satisfying Assumption \ref{Assumption}.  Then there exists a positive constant $C_3$, depending only on the space $G/K$, such that

\[ Vol_{\alpha_I}(r)\geq C_3\big(\frac{\alpha_3}{\alpha_1\alpha_2}\big)^{d_3}r^{n+d_3}, \ \ 0< r\leq \alpha_1,\]  

\noindent uniformly in $\alpha_1\leq \alpha_2\leq \alpha_3$.
\end{prop}

To prove Proposition \ref{HeisenbergLower}, we need to generate arbitrary directions $Z\in \fr{m}_3$ from directions in $\fr{m}_1, \fr{m}_2$. Set $N:=d_1d_2$.  For $k=1,\dots, d_3$, define the maps $H^k:\mathbb R^2\rightarrow G$ by the following product of commutators

\begin{equation*}H^k(s,t)=\prod_{l=1}^N\exp{(-sA^k_l)}\exp{(-tB^k_l)}\exp{(sA^k_l)}\exp{(tB^k_l)}, \end{equation*}

\noindent where the vectors $A^k_l,B^k_l$ are given by Lemma \ref{generator}. For $X\in \fr{m}_1$, $Y\in \fr{m}_2$, $Z=(z_k)_{k=1}^{d_3}\in \fr{m}_3$ and $\delta \in \mathbb R$, define the function $F:\fr{m}\times \mathbb R\rightarrow G/K$, with 

\begin{equation*}F(X,Y,Z,\delta)=\pi\big( \exp{X}\cdot \exp{Y}\cdot \prod_{k=1}^{d_3}{H^k(\op{sgn}(z_k)\sqrt{|z_k|},\delta \sqrt{|z_k|})} \big).\end{equation*}

 \noindent  For $\delta\in [0,1]$, we also consider the function $F^{\delta}:\fr{m}\rightarrow G/K$, with

\[  F^{\delta}(X,Y,Z)=F(X,Y,Z,\delta), \ X\in \fr{m}_1,\ Y\in \fr{m}_2,\ Z\in \fr{m}_3.\]

\noindent The family of functions $F^{\delta}$ is constructed to access expensive directions in $\fr{m}_3$ through the cheaper directions in $\fr{m}_1,\fr{m}_2$. We remark that $F$ generalizes a similar function defined for $SU(2)$ in \cite{EGS18}.  We first establish the following estimate.

\begin{lemma}\label{distancelemma} There exists a positive constant $C$, depending only on the space $G/K$, such that for all $X\in \fr{m}_1$, $Y\in \fr{m}_2$, $Z=(z_k)_{k=1}^{d_3}\in \fr{m}_3$ and $\delta\in [0,1]$, it holds
  
  \[ d_{\alpha_I}\big(o,F^{\delta}(X,Y,Z)\big)\leq \alpha_1\|X\|+\alpha_2\|Y\|+\frac{C}{2}(\alpha_1+\delta \alpha_2)\sum_{k=1}^{d_3}{\sqrt{|z_k|}}. \] 
  
\end{lemma}

\begin{proof} We set $M_0:=\max\{\|A^k_l\|,\|B^k_l\| : 1\leq k\leq d_3, \ 1\leq l\leq N\}$. By relations \eqref{triangle} and \eqref{triangle2}, for each $k=1,\dots ,d_3$, we have 
 
 \[ d_{\alpha_I}\big(o,\pi(H^k(s,t))\big)\leq NM_0(2\alpha_1|s|+2\alpha_2|t|). \]

 \noindent Putting $s:=\op{sgn}(z_k)\sqrt{|z_k|}$ and $t:=\delta \sqrt{|z_k|}$, the above relation yields

\[ d_{\alpha_I}\big(o,\pi(H^k(\op{sgn}(z_k)\sqrt{|z_k|},\delta \sqrt{|z_k|}))\big)\leq 2NM_0(\alpha_1+\delta\alpha_2) \sqrt{|z_k|}, \]

\noindent for all $k=1,\dots, d_3$. Putting $C:=4NM_0$, the lemma follows from the definition of $F^{\delta}$ and relations \eqref{triangle} and \eqref{triangle2}. \end{proof}

The following proposition shows the existence of a common neighborhood on which the maps $F^{\delta}$, $\delta\in(0,1]$, are $\mathcal{C}^1$-diffeomorphisms, with a uniform lower bound for their Jacobians. Along with Lemma \ref{distancelemma}, it will yield Proposition \ref{HeisenbergLower}.

\begin{prop}\label{Vprime} There exists an open ball $V$ with respect to $\left.Q\right|_{\fr{m}\times \fr{m}}$, centered at $0\in \fr{m}$, such that for all $\delta\in (0,1]$, the restriction $\left.F^{\delta}\right|_{V}$ is a $\mathcal{C}^1$ diffeomorphism onto its image and $J_{F^{\delta}}\geq\frac{1}{2} \delta^{d_3}$ on $V$. As a result, it holds 
 
 \[ \mu_0(F^{\delta}(S))\geq \frac{1}{2}\delta^{d_3} m(S),\]    

\noindent for any measurable $S\subseteq V$.\end{prop}

Before proving the above proposition, we establish sufficient regularity properties for $F^{\delta}$ through the following Lemma \ref{CBH} and Lemma \ref{PartialSmoothness}. 
 
 \begin{lemma}\label{CBH} There exists a neighborhood $U_0$ of $0$ in $\mathbb R^2$ such that, for all $k=1,\dots ,d_3$ and $(s,t)\in U_0$, it holds

\begin{equation*}H^k(s,t)=\exp(sth^k(s,t)), \end{equation*}

\noindent where $h^k:U_0\rightarrow \fr{g}$ is $\mathcal{C}^{\infty}$ and $h^k(0,0)=E_3^k\in \fr{m}_3$.
    
\end{lemma}

\begin{proof}Set 

\[ C^k_l(s,t):=\exp{(-sA^k_l)}\exp{(-tB^k_l)}\exp{(sA^k_l)}\exp{(tB^k_l)}, \]
 
 \noindent so that $H^k(s,t)=\prod_{l=1}^N{C_l^k(s,t)}$.  The Campbell-Baker-Hausdorff formula yields

\begin{eqnarray*}C^k_l(s,t)&=&\exp{\big(st[A^k_l,B^k_l]-\frac{s^2t}{2}[A^k_l,[A^k_l,B^k_l]]-\frac{st^2}{2}[B^k_l,[A^k_l,B^k_l]]+\cdots)\big)}\\ \\
&=&\exp{\big(st[A^k_l,B^k_l]+R^k_l(s,t) \big)}, \end{eqnarray*}

\noindent for $(s,t)$ in a sufficiently small neighborhood of $0\in \mathbb R^2$, where each non-zero term in $R^k_l(s,t)$ has a coefficient $s^nt^m$ with $m,n\geq 1$ and $m+n\geq 3$. By shrinking the neighborhood if necessary and multiplying the above terms for $l=1,\dots,N$, an iterated use of the Campbell-Baker-Hausdorff formula implies that there exists a neighborhood $U_k$ of $0 \in \mathbb R^2$ such that, for all $(s,t)\in U_k$, it holds  

\begin{eqnarray*}H^k(s,t)&=&\exp{\big(st\sum_{l=1}^N{[A^k_l,B^k_l]}+R^k(s,t)\big)}\\ \\
&=&\exp{\big(stE_3^k+R^k(s,t)\big)}. \end{eqnarray*}

\noindent Here, the second equality follows from Lemma \ref{generator}. Moreover, it can be verified by the iterated use of the Campbell-Baker-Hausdorff formula that any non-zero term in $R^k(s,t)$ also has a coefficient $s^nt^m$ with $m,n\geq 1$ and $m+n\geq 3$. Therefore, after shrinking $U_k$ if necessary, and by factoring out the term $st$ in the argument of the exponential, we arrive to an equation of the form

\begin{equation*}H^k(s,t)=\exp{(sth^k(s,t))}, \end{equation*}

\noindent where $h^k(s,t)$ is a convergent power series.  Thus the map $ h^k:U_k\rightarrow \fr{g}$, with $(s,t)\mapsto h^k(s,t)$, is $\mathcal{C}^{\infty}$, and $h^k(0,0)=E_3^k$. Setting $U_0:=\cap_{k=1}^{d_3}{U_k}$, we have concluded the desired statement.  \end{proof}

 In the sequel, we establish the regularity properties of $F^{\delta}$ that we need. 

\begin{lemma}\label{PartialSmoothness} There exists an open ball $V$ with respect to the inner product $\left.Q\right|_{\fr{m}\times \fr{m}}$, centered at $0$ in $\fr{m}$, such that the following are true:\\

\noindent 1. The map $F^{\delta}:\fr{m}\rightarrow G/K$ is $\mathcal{C}^1$ on $V$, for every $\delta\in [0,1]$. \\
\noindent 2. It holds

\[ dF^{\delta}_0(X+Y+Z)=X+Y+\delta Z, \ \ X\in \fr{m}_1, \ Y\in \fr{m}_2, \ Z\in \fr{m}_3.\]

\noindent 3. There exist jointly continuous maps $f_k:V\times [0,1]\rightarrow T(G/K)$, $k=1,\dots,d_3$, such that

\[ \frac{\partial F^{\delta}}{\partial z_k}(v)=\delta f_k(v,\delta) \ \ \makebox{and} \ \ f_k(0,\delta)=E_3^k, \ \ \makebox{for all} \ \ v\in V, \ \ k=1,\dots,d_3, \ \delta\in [0,1].\]

\noindent 4. The map 
  
  \[ V\times [0,1]\times \fr{m}\rightarrow T(G/K), \ \ (v,\delta,W)\mapsto dF^{\delta}_v(W),\]
 \noindent is continuous.

  \end{lemma}

\begin{proof}   Since each function $z_k\mapsto \op{sgn}(z_k)\sqrt{|z_k|}$ is continuous on $\mathbb R$, and $H^k$ is smooth, $k=1,\dots, d_3$, the partial derivatives of $F$ with respect to $x_i,y_j$, $i=1,\dots, d_1$ and $j=1,\dots, d_2$, exist and are jointly continuous in $\big((x_i)_{i=1}^{d_1},(y_j)_{j=1}^{d_2},Z,\delta\big)$.   Thus the same holds for the corresponding partial derivatives of $F^{\delta}$. Moreover, observing that $H^k(0,0)=e$, $k=1,\dots,d_3$, it is not hard to verify that

\begin{equation}\label{PartialXY}\left.\frac{\partial F^{\delta}}{\partial x_i}\right|_{0}=d\pi_e(E_1^i)=E_1^i \ \  \makebox{and}  \ \ \left.\frac{\partial F^{\delta}}{\partial y_j}\right|_{0}=d\pi_e(E_2^j)=E_2^j, \end{equation}
 
 \noindent for $i=1,\dots, d_1$, $j=1,\dots, d_2$, under the identification $T_o(G/K)=\fr{m}$. 
  
  Next we calculate the partial derivatives with respect to $z_k$, $k=1,\dots,d_3$.  Let $U_0$ be as in Lemma \ref{CBH}. We may choose $\epsilon>0$ such that

\[ (\op{sgn}(z)\sqrt{|z|},\delta\sqrt{|z|})\in U_0,\]

 \noindent for all $z\in (-\epsilon,\epsilon)$ and $\delta\in [0,1]$. Hence by Lemma \ref{CBH}, we have
 
 \begin{equation*} H^k(\op{sgn}(z_k)\sqrt{|z_k|},\delta\sqrt{|z_k|})=\exp{\big(\delta z_kh^k(\op{sgn}(z_k)\sqrt{|z_k|},\delta\sqrt{|z_k|})\big)}, \end{equation*}

 \noindent for all $z_k\in (-\epsilon, \epsilon)$, $k=1,\dots,d_3$, and $\delta\in [0,1]$. We set 
 
 \begin{equation*} a^k(z_k,\delta):= z_kh^k(\op{sgn}(z_k)\sqrt{|z_k|},\delta\sqrt{|z_k|}), \end{equation*}
 
 \noindent so that
 
 \begin{equation}\label{Hka}H^k(\op{sgn}(z_k)\sqrt{|z_k|},\delta\sqrt{|z_k|})=\exp{(\delta a^k(z_k,\delta))}.\end{equation}

 \noindent Let $V$ be an open ball with respect to the distance induced by the inner product $\left.Q\right|_{\fr{m}\times \fr{m}}$, centered at $0$ in $\fr{m}$, such that

 \[ V\subset (-\epsilon,\epsilon)^n. \]
 
 \noindent For $0\neq z_k\in (-\epsilon,\epsilon)$, we have 

 \begin{eqnarray*}\frac{\partial a^k}{\partial z_k}(z_k,\delta)&=& h^k(\op{sgn}(z_k)\sqrt{|z_k|},\delta\sqrt{|z_k|})\\ \\
 &+&\frac{1}{2}\op{sgn}(z_k)\sqrt{|z_k|}\left.\frac{\partial h^k}{\partial s}\right|_{(s,t)=(\op{sgn}(z_k)\sqrt{|z_k|},\delta\sqrt{|z_k|})}\\ \\
 &+&\frac{\delta}{2}\sqrt{|z_k|}\left.\frac{\partial h^k}{\partial t}\right|_{(s,t)=(\op{sgn}(z_k)\sqrt{|z_k|},\delta\sqrt{|z_k|})}.\end{eqnarray*} 

\noindent Since $h^k$ is $\mathcal{C}^{\infty}$ near $(0,0)$ (Lemma \ref{CBH}), the last two terms tend to zero as $z_k$ tends to zero, uniformly for $\delta\in [0,1]$, while the first term tends to $h^k(0,0)= E_3^k$, uniformly for $\delta\in [0,1]$. Hence,

\begin{equation*}\lim_{z_k\rightarrow 0}{\frac{\partial a^k}{\partial z_k}(z_k,\delta)}=E_3^k,\end{equation*}

\noindent uniformly for $\delta\in[0,1]$. On the other hand, from the definition of $a^k$, we have

\[\frac{a^k(z_k,\delta)-a^k(0,\delta)}{z_k}=h^k(\op{sgn}(z_k)\sqrt{|z_k|},\delta\sqrt{|z_k|}),\] 

\noindent and thus

\begin{equation}\label{E3Deriv}\left.\frac{\partial a^k}{\partial z_k}\right|_{(0,\delta)}=E_3^k, \ \ \delta \in [0,1], \ \ k=1,\dots,d_3.\end{equation}

\noindent Hence $\frac{\partial a^k}{\partial z_k}$ is jointly continuous on $(-\epsilon,\epsilon) \times [0,1]$. Therefore, Equation \eqref{Hka} yields

\begin{equation}\label{partial1}\frac{\partial H^k(\op{sgn}(z_k)\sqrt{|z_k|},\delta\sqrt{|z_k|})}{\partial z_k}= d\exp_{\delta a^k(z_k,\delta)}\bigg(\frac{\partial (\delta a^k)}{\partial z_k}\bigg)=\delta d\exp_{\delta a^k(z_k,\delta)}\bigg(\frac{\partial a^k}{\partial z_k}\bigg). \end{equation}

\noindent In particular, the last term is a jointly continuous function of $z_k$ and $\delta$. Recall the left and right translations $L_x,R_x$ for $x\in G$, and set 

\[ P_k:=\exp{X}\exp{Y}\prod_{m=1}^{k-1}\exp{(\delta a^m(z_m,\delta))}, \ \  \makebox{and} \ \ Q_k:=\prod_{l=k+1}^{d_3}\exp{(\delta a^l(z_l,\delta))}.\]

\noindent By Equations \eqref{Hka}, \eqref{partial1} and the definition of $F^{\delta}$, we obtain

 \begin{equation}\label{PartialZZ}\frac{\partial F^{\delta}}{\partial z_k}=d\big(\pi\circ L_{P_k}\circ R_{Q_k}\big)_{\exp{(\delta a^k(z_k,\delta))}}\bigg(\delta d\exp_{\delta a^k(z_k,\delta)}\bigg(\frac{\partial a^k}{\partial z_k}\bigg)\bigg)=\delta f_k(X,Y,Z,\delta),\end{equation}

\noindent where we have set 

\[f_k(X,Y,Z,\delta):=d\big(\pi\circ L_{P_k}\circ R_{Q_k}\circ \exp \big)_{\delta a^k(z_k,\delta)}\bigg(\frac{\partial a^k}{\partial z_k}\bigg). \]

\noindent Since $P_k,Q_k,a^k$ and $\frac{\partial a^k}{\partial z_k}$ are jointly continuous in $V\times [0,1]$, and the maps $\pi,L_x,R_x,\exp$ are smooth, it follows that the maps $f_k$ are jointly continuous in $V\times [0,1]$. Moreover, at $(X,Y,Z)=0$ we have $P_k=Q_k=e$ and $a^k(0,\delta)=0$.  Along with relation \eqref{E3Deriv}, the definition of $f_k$, and under the identification $T_o(G/K)=\fr{m}$, we obtain

\[
f_k(0,\delta)=d\pi_e\big(d\exp_0(E_3^k)\big)=d\pi_e(E_3^k)=E_3^k, \ \ k=1,\dots,d_3.
\]

\noindent The above relation, along with equations \eqref{PartialXY} and \eqref{PartialZZ}, yields parts 2. and 3. Finally, since all partial derivatives of $F^{\delta}$ exist and depend continuously on $(v,\delta)$, $v\in V, \delta\in [0,1]$, parts 1. and 4. follow. \end{proof}

We proceed to prove Proposition \ref{Vprime}. \\

\noindent \emph{Proof of Proposition \ref{Vprime}.} Let $V$ be the open ball given by Lemma \ref{PartialSmoothness} and let $\Phi:U\subset \fr{m}\rightarrow \Phi(U)\subset G/K$ be the coordinates introduced in Section \ref{sectioncoordinate}.  Since $F^{\delta}$ is jointly continuous, $F^{\delta}(0,0,0)=o$ and $[0,1]$ is compact, by further shrinking the ball $V$ if necessary, we may assume that 

\[F^{\delta}(V)\subseteq \Phi(U), \ \ \makebox{for all} \ \ \delta\in [0,1]. \]

\noindent For $\delta\in [0,1]$, we consider the function $G^{\delta}:V\rightarrow U$ with 

\[ G^{\delta}:=\Phi^{-1}\circ F^{\delta}.\]

\noindent Since $d\Phi_0=\op{Id}_{\fr{m}}$, part 1. of Lemma \ref{PartialSmoothness} yields 

\[ dG^{\delta}_0(X+Y+Z)=X+Y+\delta Z,  \ \ X\in \fr{m}_1, \ Y\in \fr{m}_2, \ Z\in \fr{m}_3, \ \delta \in [0,1]. \]

\noindent For $\delta\in (0,1]$, consider the linear map $\lambda_\delta:\fr{m}\rightarrow\fr{m}$ defined by

\[ \lambda_\delta(X+Y+Z)=X+Y+\delta^{-1}Z, \ \ X\in \fr{m}_1, \ Y\in \fr{m}_2, \ Z\in \fr{m}_3. \]

\noindent Then $\lambda_\delta^{-1}(V)\subseteq V$, given that $V$ is an open ball and $\delta \in (0,1]$. Consider the map $\tilde{G}^{\delta}:\lambda_\delta^{-1}(V)\rightarrow U$, with 

\[ \tilde{G}^{\delta}:=G^{\delta}\circ \lambda_\delta=\Phi^{-1}\circ F^{\delta}\circ \lambda_\delta, \ \ \delta \in (0,1]. \]

\noindent Observe that $d\tilde{G}^{\delta}_0=\op{Id}_{\fr{m}}$.  We claim that, after shrinking the ball $V$ if necessary, it holds 

\begin{equation}\label{ClaimGd}\|d\tilde{G}^{\delta}_u-\op{Id}_{\fr{m}}\|_{op}<\frac{1}{2}, \end{equation}

\noindent for all $u\in \lambda_\delta^{-1}(V)$ and $\delta\in (0,1]$, where $\| \cdot \|_{op}$ denotes the operator norm in $\fr{gl}(\fr{m})$.

To prove our claim, let $u=\lambda_\delta^{-1}(v)$, $v\in V$, and  $W=X+Y+Z\in \fr{m}$, where $X\in \fr{m}_1$, $Y\in \fr{m}_2$ and $Z=(z_k)_{k=1}^{d_3}\in \fr{m}_3$. Using the chain rule and part 3. of Lemma \ref{PartialSmoothness}, we obtain

\[ d\tilde{G}^{\delta}_uW=d\Phi^{-1}_{F^{\delta}(v)}\big(dF^{\delta}_v(X+Y)+\sum_{k=1}^{d_3}{z_k f_k(v,\delta)}\big).\]

\noindent In view of the above relation, we define the map $A:V\times [0,1]\rightarrow \fr{gl}(\fr{m})$ by 

\[ A(v,\delta)W:=d\Phi^{-1}_{F^{\delta}(v)}\big(dF^{\delta}_v(X+Y)+\sum_{k=1}^{d_3}{z_k f_k(v,\delta)}\big), \]

\noindent where $v\in V$, $\delta \in [0,1]$ and $W=X+Y+Z\in \fr{m}$.  The previous relation shows that $A(v,\delta)=d\tilde{G}^{\delta}_u$, where $u=\lambda_\delta^{-1}(v)$, for all $\delta \in (0,1]$. Moreover, by parts 1. and 4. of Lemma \ref{PartialSmoothness}, the map $A$ is continuous, and $A(0,\delta)=d\Phi^{-1}_o\circ \left.d\pi_e\right|_{\fr{m}}=\op{Id}_{\fr{m}}$ for all $\delta \in [0,1]$. Since $[0,1]$ is compact, by shrinking the ball $V$ if necessary, we obtain

\[ \|A(v,\delta)-\op{Id}_{\fr{m}}\|_{op}<\frac{1}{2}, \ \ \makebox{i.e.,} \ \ \|d\tilde{G}^{\delta}_u-\op{Id}_{\fr{m}}\|_{op}<\frac{1}{2}, \]

\noindent for all $\delta\in (0,1]$ and $u\in \lambda_\delta^{-1}(V)\subseteq V$, which proves relation \eqref{ClaimGd}.

Since the open ball $V$ is convex and $\lambda_\delta^{-1}$ is linear, then $\lambda_\delta^{-1}(V)$ is also convex and thus the mean value theorem (see for example \cite{Rudin}, Theorem 9.19) implies

\[  \|\tilde{G}^{\delta}(u_1)-\tilde{G}^{\delta}(u_2)-(u_1-u_2)\|\leq \frac{1}{2}\|u_1-u_2\|, \ \ u_1,u_2\in \lambda_\delta^{-1}(V).   \]

\noindent Therefore,  

\[ \|\tilde{G}^{\delta}(u_1)-\tilde{G}^{\delta}(u_2)\|\geq \frac{1}{2}\|u_1-u_2\|,  \]

\noindent which shows that $\tilde{G}^{\delta}$ is injective for all $\delta \in(0,1]$.  Since $\tilde{G}^{\delta}=\Phi^{-1}\circ F^{\delta}\circ \lambda_\delta$ and $\lambda_\delta$ is an isomorphism from $\lambda_\delta^{-1}(V)$ onto $V$, $\delta \in (0,1]$, we deduce that $F^{\delta}$ is injective on $V$ for all $\delta \in (0,1]$. 

Now we estimate the Jacobian of $F^{\delta}$.  Consider the Riemannian volume form $\omega_0$ on $\Phi(U)\subset G/K$ with respect to the reference metric $\left.Q\right|_{\fr{m}\times \fr{m}}$, and the function $j:V\times [0,1]\rightarrow \mathbb R$ with

\[ j(X,Y,Z,\delta)=\bigg|\omega_0\bigg( \big(\frac{\partial F^{\delta}}{\partial x_i}\big)_{i=1}^{d_1}, \big(\frac{\partial F^{\delta}}{\partial y_j}\big)_{j=1}^{d_2}, (f_k)_{k=1}^{d_3}\bigg)\bigg|. \]

\noindent By Lemma \ref{PartialSmoothness}, the function $j$ is jointly continuous, and $j(0,0,0,\delta)=1$ for all $\delta\in [0,1]$.  Hence, after shrinking the ball $V$ if necessary, and since $[0,1]$ is compact, we have

\[ j(X,Y,Z,\delta)\geq \frac{1}{2}, \ \ (X,Y,Z,\delta)\in V\times [0,1]. \]

\noindent For $\delta\in (0,1]$, Lemma \ref{PartialSmoothness} yields $\frac{\partial F^{\delta}}{\partial z_k}=\delta f_k$, $k=1,\dots,d_3$, and hence 

\[ \bigg|\omega_0\bigg( \big(\frac{\partial F^{\delta}}{\partial x_i}\big)_{i=1}^{d_1}, \big(\frac{\partial F^{\delta}}{\partial y_j}\big)_{j=1}^{d_2}, \big(\frac{\partial F^{\delta}}{\partial z_k}\big)_{k=1}^{d_3}\bigg)\bigg|=\delta^{d_3}j(X,Y,Z,\delta)\geq \frac{1}{2}\delta^{d_3}>0. \]

\noindent Therefore, $dF^{\delta}$ is non-singular on $V$ for all $\delta \in (0,1]$. Since $F^{\delta}$ is also injective on $V$ for all $\delta \in (0,1]$, the inverse function theorem implies that $\left.F^{\delta}\right|_{V}$ is a $\mathcal{C}^1$-diffeomorphism onto its image for all $\delta \in (0,1]$.  Thus the Jacobian $J_{F^{\delta}}$ is well-defined on $V$ in the sense of Section \ref{Jacobian}, while $J_{F^{\delta}}=\delta^{d_3}j\geq \frac{1}{2}\delta^{d_3}$ on $V$, for all $\delta \in (0,1]$.  We conclude that for any measurable subset $S\subseteq V$, it holds

\[ \mu_0(F^{\delta}(S))=\int_{S}{J_{F^{\delta}}dm}\geq \frac{1}{2}\delta^{d_3}m(S), \]

\noindent which completes the proof. \qed   \\
 
We are now ready to establish the lower bound of Proposition \ref{HeisenbergLower}. \\

\noindent \emph{Proof of Proposition \ref{HeisenbergLower}.} We choose $\epsilon\in (0,1)$ such that $[-\epsilon,\epsilon]^n\subset V$, where $V$ is the neighborhood of $0\in \fr{m}$ as in Proposition \ref{Vprime}. For $X=(x_i)_{i=1}^{d_1}\in \fr{m}_1$, $Y=(y_j)_{j=1}^{d_2}\in \fr{m}_2$ and $Z=(z_k)_{k=1}^{d_3}\in \fr{m}_3$, and putting $\delta:=\frac{\alpha_1}{\alpha_2}\in (0,1]$, Lemma \ref{distancelemma} asserts that there exists a positive constant $C$, depending only on the space $G/K$, such that 

  \begin{equation}\label{LemmaDistanceRelation} d_{\alpha_I}\big(o,F^{\frac{\alpha_1}{\alpha_2}}(X,Y,Z)\big)\leq \alpha_1\|X\|+\alpha_2\|Y\|+C\alpha_1\sum_{k=1}^{d_3}{\sqrt{|z_k|}}. \end{equation}

\noindent Assume that $0< r\leq \epsilon\alpha_1$.  We set $C^{\prime}:=2(\sqrt{d_1}+\sqrt{d_2}+Cd_3)$, and we consider the box

\[ S_r:=\bigg[-\frac{r}{C^{\prime}\alpha_1}, \frac{r}{C^{\prime}\alpha_1}\bigg]^{d_1}\times \bigg[-\frac{r}{C^{\prime}\alpha_2}, \frac{r}{C^{\prime}\alpha_2}\bigg]^{d_2}\times \bigg[-\big(\frac{r}{C^{\prime}\alpha_1}\big)^2, \big(\frac{r}{C^{\prime}\alpha_1}\big)^2\bigg]^{d_3}.\]

\noindent Since $r\leq \epsilon \alpha_1\leq \epsilon \alpha_2$ and $\epsilon \in (0,1)$, we have 

 \[ \frac{r}{C^{\prime}\alpha_1}\leq \epsilon, \ \ \frac{r}{C^{\prime}\alpha_2}\leq \epsilon \ \ \makebox{and}\ \ (\frac{r}{C^{\prime}\alpha_1})^2\leq \epsilon, \]
 
 \noindent and thus $S_r\subseteq [-\epsilon, \epsilon]^n$.  Moreover, 
 
 \[ m(S_r)=C^{\prime \prime}(\alpha_1^{d_1+2d_3}\alpha_2^{d_2})^{-1}r^{n+d_3},\ \ \makebox{with} \ \  C^{\prime \prime}:=\frac{2^n}{(C^{\prime})^{n+d_3}}.\]

\noindent Besides, for any $(X,Y,Z)\in S_r$, with $Z=(z_k)_{k=1}^{d_3}$, we have

\[ \|X\|\leq \frac{\sqrt{d_1}r}{C^{\prime}\alpha_1}, \ \ \|Y\|\leq \frac{\sqrt{d_2}r}{C^{\prime}\alpha_2} \ \ \makebox{and}\ \ \sum_{k=1}^{d_3}\sqrt{|z_k|}\leq \frac{d_3r}{C^{\prime}\alpha_1}. \]

\noindent Thus relation \eqref{LemmaDistanceRelation} yields

\[ d_{\alpha_I}\big(o,F^{\frac{\alpha_1}{\alpha_2}}(X,Y,Z)\big)\leq \frac{r}{C^{\prime}}\big(\sqrt{d_1}+\sqrt{d_2}+Cd_3\big)<r, \]

\noindent and hence $F^{\frac{\alpha_1}{\alpha_2}}(S_r)\subseteq B_{\alpha_I}(o,r)$. Along with Proposition \ref{Vprime}, we obtain
  
 \[ V_{\alpha_I}(r)=\mu_0(B_{\alpha_I}(o,r))\geq \mu_0(F^{\frac{\alpha_1}{\alpha_2}}(S_r))\geq \frac{1}{2}\big(\frac{\alpha_1}{\alpha_2}\big)^{d_3}m(S_r),\]  

\noindent that is,

\[ V_{\alpha_I}(r)\geq c(\alpha_1^{d_1+d_3}\alpha_2^{d_2+d_3})^{-1}r^{n+d_3}, \ \ 0< r\leq \epsilon\alpha_1, \]

\noindent where $c:= \frac{C^{\prime \prime}}{2}$ depends only on the space $G/K$.  Finally, for $\epsilon \alpha_1 \leq r\leq \alpha_1$, the monotonicity of the volume function yields

\[ V_{\alpha_I}(r)\geq V_{\alpha_I}(\epsilon\alpha_1)\geq c \epsilon^{n+d_3}\big(\frac{\alpha_1}{\alpha_2}\big)^{d_2+d_3}.\]

\noindent Given that $\alpha_1 \geq r$, the above relation yields
\[ V_{\alpha_I}(r)\geq c\epsilon^{n+d_3}(\alpha_1^{d_1+d_3}\alpha_2^{d_2+d_3})^{-1}r^{n+d_3}, \ \  \epsilon \alpha_1 \leq r\leq \alpha_1. \]

\noindent Taking $C_3:=c \min{\{1, \epsilon^{n+d_3}\}}$, it follows that
 
 \[ V_{\alpha_I}(r)\geq C_3(\alpha_1^{d_1+d_3}\alpha_2^{d_2+d_3})^{-1}r^{n+d_3}, \ \  0<r\leq \alpha_1, \]

 \noindent while multiplying by $\Lambda_{\alpha_I}$ yields the desired result. \qed

\subsection{Sub-Riemannian upper bound}\label{HeisenbergUpperSection}
In this subsection, we obtain the following analogous upper bound for the sub-Riemannian interval of the radius.

\begin{prop}\label{UpperHeisenbergProp} Let $(G/K,\alpha_I)$ be a homogeneous space satisfying Assumption \ref{Assumption}.  Then there exist positive constants $\eta\in (0,1)$ and $C_4$, depending only on the space $G/K$, such that

\begin{equation*}Vol_{\alpha_I}(r)\leq C_4\big(\frac{\alpha_3}{\alpha_1\alpha_2}\big)^{d_3}r^{n+d_3}, \ \ \frac{\alpha_1\alpha_2}{\alpha_3}\leq r \leq \eta \alpha_1, \end{equation*}

\noindent uniformly in $\alpha_1\leq \alpha_2\leq\alpha_3$.
\end{prop}

\begin{remark} The interval $\frac{\alpha_1\alpha_2}{\alpha_3}\leq r \leq \eta \alpha_1$ may be empty for some choices of $\alpha_1\leq\alpha_2\leq\alpha_3$, in which case Proposition \ref{UpperHeisenbergProp} holds trivially.
\end{remark}

  Before we prove Proposition \ref{UpperHeisenbergProp}, we need the following setup. Fix a neighborhood $V_0$ of $0$ in $\fr{g}$ such that $\exp:V_0\rightarrow \exp{V_0}\subseteq G$ is a diffeomorphism.  Let $\gamma:[0,1]\rightarrow \exp{V_0}$ be a smooth curve with $\gamma(0)=e$, and consider its left Darboux derivative (see e.g. \cite{Sharpe}, Chapter 3, Section 5)

\begin{equation}\label{MaurerCartanDerivative} X(t):=(dL_{\gamma(t)^{-1}})_{\gamma(t)}{\dot{\gamma}(t)}\in \fr{g}.  \end{equation}

\noindent Write

\begin{equation}\label{Darbo} X(t)=X_0(t)+X_1(t)+X_2(t)+X_3(t) \in \fr{g} \ \  \makebox{with} \ \ X_0(t)\in \fr{k} \ \  \makebox{and}  \ \ X_i(t)\in \fr{m}_i,\ \ i=1,2,3. \end{equation}

\noindent We extend $X(t)$ to the left-invariant, time-dependent vector field $X^L$ in $G$, given by 

\[ X^L(t,y)=(dL_y)_e(X(t)), \ \ y\in G,\ \ t\in [0,1],\]
 
\noindent with $X^L(t,e)=X(t)$. Observe that $\gamma$ satisfies the controlled differential equation 

\begin{equation}\label{CDE} \dot{\gamma}(t)=X^L(t,\gamma(t)),\ \  \gamma(0)=e.\end{equation}

  \noindent Since the metric induced by the inner product $Q$ on $G$ is left-invariant, the vector field $X^L$ has constant length equal to $\|X(t)\|$, $t\in [0,1]$, with respect to this metric. We identify $\fr{g}$ with the Lie algebra of left-invariant vector fields on $G$, and we apply the main theorem of Strichartz (\cite{Str87}) on Equation \eqref{CDE}. Since $\fr{g}$ is finite-dimensional, by the condition for convergence in \cite{Str87} (see p. 335), we obtain the following. 

\begin{lemma}\label{ChenStrichatzLemma}  There exists $\epsilon>0$ such that if $\sup_{t\in [0,1]}\|X(t)\|\leq \epsilon$ and $\gamma(1)=\exp{Z}$, with $Z\in V_0$, then 

\[ Z=\sum_{k=1}^{\infty}{Z^k},   \]

 \noindent with 

\[ Z^k=\sum_{\sigma \in S_k}{\bigg(\frac{(-1)^{e(\sigma)}}{k^2\binom{k-1}{e(\sigma)}}\bigg)\int_{\Delta_k}{\big[[\cdots [X(t_{\sigma(1)}),X(t_{\sigma(2)})], \cdots ], X(t_{\sigma(k)})]\big]dt_1dt_2\cdots dt_k}},    \]

\noindent where for $k=1$, the iterated bracket in the integral is simply $X(t_1)$. 

Here $\Delta_k$ denotes the $k$-simplex $\{(t_1,\dots,t_k)\in \mathbb R^k: 0\leq t_1\leq \cdots \leq t_k\leq 1\}$, $S_k$ denotes the permutation group of $\{1,\dots,k\}$ and $e(\sigma)$ denotes the number of errors in ordering consecutive terms in $\{\sigma(1),\dots,\sigma(k)\}$, that is $e(\sigma)=|\{i\in \{1,\dots,k-1\}: \sigma(i+1)<\sigma(i)\}|$. 

In particular, in view of relation \eqref{Darbo},

\begin{equation}\label{ChenExpression} Z^k=\sum_{\sigma \in S_k}{\bigg(\frac{(-1)^{e(\sigma)}}{k^2\binom{k-1}{e(\sigma)}}\bigg)\sum_{i_1,\dots, i_k\in \{0,1,2,3\}}{\int_{\Delta_k}{\big[[\cdots [X_{i_1}(t_{\sigma(1)}),X_{i_2}(t_{\sigma(2)})], \cdots ], X_{i_k}(t_{\sigma(k)})]\big]dt_1dt_2\cdots dt_k}}}.    \end{equation}

\end{lemma}

To use the above formula effectively, we need two auxiliary results.

\begin{lemma}\label{iterated brackets}Consider an iterated bracket of the form

\[ I:=\big[[\cdots [X_{i_1}(t_{\sigma(1)}),X_{i_2}(t_{\sigma(2)})], \cdots ], X_{i_k}(t_{\sigma(k)})]\big], \]

\noindent where $i_1,\dots,i_k\in \{0,1,2,3\}$. Then the following are true:\\

\noindent \emph{i.} If $I$ has non-zero $Q$-orthogonal projection on $\fr{m}_2$, then at least one of the indices $i_1,\dots,i_k$ is equal to 2 or 3, i.e., at least one of the vectors $X_{i_j}(t_{\sigma(j)})$, $j=1,\dots,k$, lies in $\fr{m}_2$ or $\fr{m}_3$.

\noindent \emph{ii.} If $I$ has non-zero $Q$-orthogonal projection on $\fr{m}_3$, then either there exists at least one index among $i_1,\dots,i_k$ that is equal to 3, or there exists at least one index equal to 1 and at least one index equal to 2.
\end{lemma}

\begin{proof}  For part i., assume on the contrary that none of the indices $i_1,\dots,i_k$ is equal to 2 or 3. Then all vectors $X_{i_j}(t_{\sigma(j)})$, $j=1,\dots,k$, lie in $\fr{k}\oplus \fr{m}_1$. Under Assumption \ref{Assumption}, by part 1. of Lemma \ref{LemmaRelaxed}, the space $\fr{k}\oplus \fr{m}_1$ is a subalgebra of $\fr{g}$, and thus $I\in \fr{k}\oplus \fr{m}_1$, a contradiction.  For part ii., if none of the indices $i_1,\dots,i_k$ is equal to 3, then all vectors $X_{i_j}(t_{\sigma(j)})$, $j=1,\dots,k$, lie in $\fr{k}\oplus \fr{m}_1\oplus \fr{m}_2$. In that case, at least one index is equal to 1, for otherwise all vectors $X_{i_j}(t_{\sigma(j)})$ would lie in the subalgebra $\fr{k}\oplus \fr{m}_2$ and we would have $I\in \fr{k}\oplus \fr{m}_2$, a contradiction. Similarly, at least one index is equal to 2.   \end{proof}

\noindent We state and prove our second auxiliary result, that will allow us to bound the $\fr{k}$-component $X_0(t)$ of the Darboux derivative $X(t)$.

\begin{lemma}\label{KcomponentBound} Let $G/K$ be a homogeneous space with $G$ compact, having the reductive decomposition $\fr{g}=\fr{k}\oplus\fr{m}$ with respect to an $\op{Ad}$-invariant inner product $Q$ on $\fr{g}$, and let $\| \cdot \|$ denote the norm induced by $Q$. There exists a neighborhood $V_{\fr{m}}$ of $0$ in $\fr{m}$ and a positive constant $C_{\fr{m}}$, depending only on the space $G/K$ and the reductive decomposition, such that the following hold:\\
 \noindent 1. The map $\Psi:V_{\fr{m}}\rightarrow \pi(\exp{V_{\fr{m}}})$, $v\mapsto \pi(\exp{v})$ is a diffeomorphism. \\
 
 \noindent 2. Let $v:[0,1]\rightarrow V_{\fr{m}}$ be a smooth curve with $v(0)=0$, and set $\gamma(t):=\exp{v(t)}$. Let
 
 \[ X(t)=(dL_{\gamma(t)^{-1}})_{\gamma(t)}{\dot{\gamma}(t)}, \]
 
 \noindent be the Darboux derivative of $\gamma$, and decompose 
 
 \[  X(t)=X_{0}(t)+X_{\fr{m}}(t),\] 
 
 \noindent where $X_{0}(t)\in \fr{k}$ and $X_{\fr{m}}(t)\in \fr{m}$.  Then 
  
  \[ \|\dot{v}(t)\|\leq C_{\fr{m}}\|X_{\fr{m}}(t)\| \ \ \makebox{and} \ \ \|X_0(t)\|\leq C_{\fr{m}}\|X_{\fr{m}}(t)\|, \ \ \makebox{for all} \ \ t\in [0,1].  \]
  \end{lemma}

\begin{proof} There exists a neighborhood $U$ of $0$ in $\fr{m}$ such that the map $\left.\Psi\right|_{U}$ is a diffeomorphism onto its image (c.f. Section \ref{sectioncoordinate}).  The derivative of the exponential map (see e.g. \cite{Hel}, p. 105, 106), yields
   
\begin{eqnarray*} X(t)&=&(dL_{\exp{(-v(t))}})_{\exp{v(t)}}\circ (d\exp_{v(t)})(\dot{v}(t))\\ \\ 
  &=&\frac{1-e^{-\op{ad}_{v(t)}}}{\op{ad}_{v(t)}}(\dot{v}(t))=\bigg(\sum_{k=0}^{\infty}{\frac{(-1)^k}{(k+1)!}\op{ad}_{v(t)}^k}\bigg)(\dot{v}(t)).\end{eqnarray*}

\noindent Now let $\pi_{\fr{k}}:\fr{g}\rightarrow \fr{k}$ and $\pi_{\fr{m}}:\fr{g}\rightarrow \fr{m}$ be the corresponding $Q$-orthogonal projections. Moreover, for $v\in \fr{m}$, define the linear map $P_v:\fr{m}\rightarrow \fr{m}$ by 

\begin{equation}\label{analyticmap}
P_v:=\pi_{\fr{m}}\circ
\left.\bigg(
\frac{1-e^{-\op{ad}_v}}{\op{ad}_v}\bigg)\right|_{\fr{m}}=\sum_{k=0}^{\infty}{\frac{(-1)^k}{(k+1)!}
\pi_{\fr{m}}\circ
\left.\op{ad}_v^k\right|_{\fr{m}}}.
\end{equation}

\noindent Then  

\begin{equation}\label{Xinvert}X_{\fr{m}}(t)=P_{v(t)}(\dot{v}(t)), \ \ t\in [0,1].\end{equation}
   
By relation \eqref{analyticmap}, the map $\mathcal{P}:\fr{m}\rightarrow \fr{gl}(\fr{m})$,  $v\mapsto P_v$, is analytic. Since $\mathcal{P}(0)=P_0=\op{Id}_{\fr{m}}\in \op{GL}(\fr{m})$, and $\op{GL}(\fr{m})$ is open in $\fr{gl}(\fr{m})$, there exists a neighborhood $W\subseteq U$ of $0$ in $\fr{m}$ such that $P_v\in \op{GL}(\fr{m})$ for all $v\in W$. Choose an open ball $V_{\fr{m}}$ with respect to $\left.Q\right|_{\fr{m}\times \fr{m}}$, centered at $0\in \fr{m}$, such that $V_{\fr{m}}\subset \overline{V_{\fr{m}}}\subset W$.  Observe that the ball $V_{\fr{m}}$ does not depend on the choice of the curve $v(t)$. Since $v\mapsto P_v^{-1}$ is continuous and $\overline{V_{\fr{m}}}$ is compact, there exists a positive constant $C$ such that $\|P_v^{-1}\|_{op}\leq C$ for all $v\in \overline{V_{\fr{m}}}$, where $\| \cdot \|_{op}$ denotes the operator norm in $\fr{gl}(\fr{m})$. Equation \eqref{Xinvert} then yields

 \begin{equation}\label{LemmaBFirstBound} \|\dot{v}(t)\|\leq \|P_{v(t)}^{-1}\|_{op}\|X_{\fr{m}}(t)\|\leq C \|X_{\fr{m}}(t)\|, \ \ \makebox{for all} \ \ t\in [0,1]. \end{equation}

To estimate $X_0(t)$, write

\begin{equation*}X_0(t)=\pi_{\fr{k}}(X(t))=\pi_{\fr{k}}\bigg(\sum_{k=0}^{\infty}{\frac{(-1)^k}{(k+1)!} \op{ad}_{v(t)}^k(\dot{v}(t)})\bigg).    \end{equation*}

\noindent  Let $C_{\fr{g}}$ be the constant defined by relation \eqref{Cg}, so that the above equation yields

\[  \|X_0(t)\|\leq \|X(t)\|\leq \bigg(\sum_{k=0}^{\infty}{\frac{\big(C_{\fr{g}}\|v(t)\|\big)^k}{(k+1)!}}\bigg)\|\dot{v}(t)\|, \ \ t\in [0,1].\]

 \noindent If $C_{\fr{g}}=0$ (i.e., $\fr{g}$ is abelian), then $X_0(t)=0$ and the result follows immediately. Otherwise, since $\overline{V_{\fr{m}}}$ is compact, set $C^{\prime}:=\max_{v\in \overline{V_{\fr{m}}}}{\|v\|}<\infty$. Then the above inequality yields

\[  \|X_0(t)\|\leq \frac{e^{C^{\prime}C_{\fr{g}}}-1}{C^{\prime}C_{\fr{g}}}\|\dot{v}(t)\|,   \]
 
\noindent which, along with relation \eqref{LemmaBFirstBound}, yields 

\[ \|X_0(t)\|\leq C \frac{e^{C^{\prime}C_{\fr{g}}}-1}{C^{\prime}C_{\fr{g}}}\|X_{\fr{m}}(t)\|.\]

\noindent Setting $C_{\fr{m}}:=C\max\{1,\frac{e^{C^{\prime}C_{\fr{g}}}-1}{C^{\prime}C_{\fr{g}}}\}$ yields the desired bounds.     
   \end{proof}

We are ready to prove Proposition \ref{UpperHeisenbergProp}.\\

\noindent \emph{Proof of Proposition \ref{UpperHeisenbergProp}.} Let $V_{\fr{m}}$ be the neighborhood of $0\in \fr{m}$ given by Lemma \ref{KcomponentBound}. By shrinking $V_{\fr{m}}$ if necessary, we may assume that $V_{\fr{m}}$ is contained in an open neighborhood $V_0$ of $0\in \fr{g}$ such that $\exp:V_0\rightarrow \exp{V_0}$ is a diffeomorphism. Let $B_{Q_{\fr{m}}}(o,\eta)$ denote the open ball in $G/K$, centered at $o$ with radius $\eta$, with respect to the distance induced from the metric $\left.Q\right|_{\fr{m}\times \fr{m}}$. We may choose $\eta$ sufficiently small so that $B_{Q_{\fr{m}}}(o,\eta)\subseteq \Psi(V_{\fr{m}})$, and since $\alpha_I\geq \left.\alpha_1^2Q\right|_{\fr{m}\times \fr{m}}$, we have

\[ B_{\alpha_I}(o,\eta\alpha_1)\subseteq B_{Q_{\fr{m}}}(o,\eta)\subseteq \Psi(V_{\fr{m}}),\]
  
\noindent uniformly in $\alpha_1\leq \alpha_2\leq \alpha_3$.   

Let $0<r\leq \eta\alpha_1$, and choose $p\in B_{\alpha_I}(o,r)\subseteq \Psi(V_{\fr{m}})$. Since the connected homogeneous space $(G/K,\alpha_I)$ is complete, we may choose a minimizing (constant speed) geodesic $f:[0,1]\rightarrow B_{\alpha_I}(o,r)$, with $f(0)=o$, $f(1)=p$ and $\op{Length}(f) <r$. Since $f([0,1])\subseteq\Psi(V_{\fr{m}})$, there exists a smooth curve $v:[0,1]\rightarrow V_{\fr{m}}$ such that $f(t)=\pi(\exp{v(t)})$, $t\in [0,1]$. We set $\gamma(t):=\exp{v(t)}$. By the definitions of $V_{\fr{m}}$ and $V_0$, we have $\gamma([0,1])\subseteq \exp(V_0)$. We consider the left Darboux derivative of $\gamma$,

\[ X(t)=X_0(t)+X_1(t)+X_2(t)+X_3(t), \]
 
 \noindent as defined by relation \eqref{MaurerCartanDerivative}.  Taking into account the fact that $\op{Length}(f) <r$ and $f$ has constant speed, the definition of the left Darboux derivative, the identity $\pi\circ L_x=\tau_x\circ \pi$, $x\in G$, and the $G$-invariance of the metric $g$, we obtain  

    \begin{eqnarray*} r^2 &>& g(\dot{f}(t),\dot{f}(t))=g \bigg((d\tau_{\gamma(t)})_{o}\circ (d\pi)_eX(t), (d\tau_{\gamma(t)})_{o}\circ (d\pi)_eX(t)\bigg)  \\  \\ 
    &=&g \big((d\pi)_eX(t), (d\pi)_eX(t)\big)=\alpha_1^2\|X_1(t)\|^2+\alpha_2^2\|X_2(t)\|^2+\alpha_3^2\|X_3(t)\|^2, \end{eqnarray*}
 
 \noindent from which we get 

\begin{equation}\label{MaurerBoundsm} \|X_i(t)\|\leq \frac{r}{\alpha_i}, \ \ i=1,2,3.\end{equation}
 
Therefore, $\|X_{\fr{m}}(t)\|\leq \frac{\sqrt{3}r}{\alpha_1}$.  Applying Lemma \ref{KcomponentBound}, we also obtain

\begin{equation}\label{MaurerBoundsk} \|X_0(t)\|\leq \sqrt{3}C_{\fr{m}}\frac{r}{\alpha_1}.\end{equation}

\noindent Since $r\leq \eta\alpha_1$, we have $\sup_{t\in [0,1]}\|X(t)\|\leq \sqrt{3}(1+C_{\fr{m}})\eta$.  Thus we can shrink $\eta$ if necessary, so that $\sup_{t\in [0,1]}\|X(t)\|\leq \epsilon$, where $\epsilon$ is given by Lemma \ref{ChenStrichatzLemma}.  Since $\gamma(1)=\exp{v(1)}$, where $v(1)\in V_{\fr{m}}\subset V_0$, we may apply Lemma \ref{ChenStrichatzLemma} with $Z=v(1)$. Since $Z\in \fr{m}$, we can write 

\[ Z=Z_1+Z_2+Z_3, \ \ Z_i\in \fr{m}_i.\]

\noindent We will use Lemma \ref{ChenStrichatzLemma} to obtain bounds for $\|Z_i\|$, $i=1,2,3$.  Let $\pi_{\fr{m}_i}:\fr{g}\rightarrow \fr{m}_i$ denote the corresponding $Q$-orthogonal projections, $i=1,2,3$.  Consider an iterated bracket of the form

\begin{equation}\label{TempBracket} I:=\big[[\cdots [X_{i_1}(t_{\sigma(1)}),X_{i_2}(t_{\sigma(2)})], \cdots ], X_{i_k}(t_{\sigma(k)})]\big], \ \ i_1,\dots,i_k\in \{0,1,2,3\}, \end{equation}

\noindent and the constant $C_{\fr{g}}$ defined by relation \eqref{Cg}. Setting $C:=\max\{1,\sqrt{3}C_{\fr{m}}\}$, relations \eqref{MaurerBoundsm} and \eqref{MaurerBoundsk} yield

\[ \|\pi_{\fr{m}_1}(I)\|\leq C_{\fr{g}}^{k-1}(C\frac{r}{\alpha_1})^k. \]

\noindent Along with the facts that there are at most $4^k$ such iterated brackets in the expression \eqref{ChenExpression}, $\bigg|\bigg(\frac{(-1)^{e(\sigma)}}{k^2\binom{k-1}{e(\sigma)}}\bigg)\bigg|\leq 1$, $|S_k|=k!$ and the volume of the simplex $\Delta_k$ is $\frac{1}{k!}$, relation \eqref{ChenExpression} yields

\[ \|\pi_{\fr{m}_1}(Z^k)\|\leq k!\int_{\Delta_k}{C_{\fr{g}}^{k-1}(4C\frac{r}{\alpha_1})^k dt_1\cdots dt_k}\leq C_{\fr{g}}^{k-1}(4C\frac{r}{\alpha_1})^k=4C\frac{r}{\alpha_1}(4C_{\fr{g}}C\frac{r}{\alpha_1})^{k-1}. \]

\noindent  We further shrink $\eta$ if necessary, so that $4C_{\fr{g}}C\eta\leq \frac{1}{2}$, and thus 

\[ 4C_{\fr{g}}C\frac{r}{\alpha_1}\leq \frac{1}{2}.\] 

\noindent Therefore, we obtain  

\begin{equation}\label{Z1Bound}\|Z_1\|\leq \sum_{k=1}^{\infty}{\|\pi_{\fr{m}_1}(Z^k)\|}\leq 4C\frac{r}{\alpha_1}\sum_{k=1}^{\infty}{(\frac{1}{2})^{k-1}}\leq 8C\frac{r}{\alpha_1}.  \end{equation}

To obtain the bound for $Z_2$, in view of Lemma \ref{iterated brackets}, if an iterated bracket of the form \eqref{TempBracket} has non-zero projection on $\fr{m}_2$, then at least one of the indices $i_1,\dots, i_k$ is equal to $2$ or $3$.  By virtue of relation \eqref{MaurerBoundsm}, and given that $\alpha_2\leq \alpha_3$, at least one of the indices $i_j$ satisfies $\|X_{i_j}(t_{\sigma(j)})\|\leq \frac{r}{\alpha_2}$.  Therefore, the same reasoning as for $Z_1$ yields $\|\pi_{\fr{m}_2}(Z^k)\|\leq 4\frac{r}{\alpha_2}(4C_{\fr{g}}C\frac{r}{\alpha_1})^{k-1}\leq 4\frac{r}{\alpha_2}(\frac{1}{2})^{k-1}$, and in turn

\begin{equation}\label{Z2Bound}\|Z_2\|\leq 8\frac{r}{\alpha_2}. \end{equation}
 
For the bound of $Z_3$, in view of Lemma \ref{iterated brackets}, if an iterated bracket of the form \eqref{TempBracket} has non-zero projection on $\fr{m}_3$, then either at least one of the indices $i_1,\dots, i_k$ is equal to $3$ or there exists at least one index equal to 1 and at least one index equal to 2. Thus for each $k\geq 1$, we can split the relevant set of contributing iterated brackets into two corresponding disjoint sets $S^k_3$ and $S^k_{12}$, where $S^k_3$ contains the $k$-fold brackets with at least one index equal to 3 and $S^k_{12}$ contains the $k$-fold brackets with no index equal to $3$, and at least one index equal to 1 and at least one index equal to 2. Observe that $S^1_{12}$ is empty.  Using the crude bound $4^k$ for the cardinality of both sets, and working in a similar way to the other two cases, we obtain

 \begin{eqnarray*}\|\pi_{\fr{m}_3}(Z^1)\|&\leq & 4\frac{r}{\alpha_3}, \ \ \makebox{and}\\ \\ 
 \|\pi_{\fr{m}_3}(Z^k)\|&\leq& 4\frac{r}{\alpha_3}(4C_{\fr{g}}C\frac{r}{\alpha_1})^{k-1}+4^2C_{\fr{g}}\frac{r}{\alpha_1}\frac{r}{\alpha_2}(4CC_{\fr{g}}\frac{r}{\alpha_1})^{k-2}\\ \\ 
 &\leq& 4\frac{r}{\alpha_3}(\frac{1}{2})^{k-1}+4^2C_{\fr{g}}\frac{r}{\alpha_1}\frac{r}{\alpha_2}(\frac{1}{2})^{k-2}, \ \ k\geq 2.\end{eqnarray*}

\noindent By summing as in the cases for $Z_1,Z_2$, it follows that

\begin{equation}\label{Z3Bound}\|Z_3\|\leq 8\big(\frac{r}{\alpha_3}+4C_{\fr{g}}\frac{r}{\alpha_1}\frac{r}{\alpha_2}\big).  \end{equation}

Setting $C^{\prime}:=8\max\{1,C,4C_{\fr{g}}\}$, relations \eqref{Z1Bound} - \eqref{Z3Bound} yield

 \begin{equation*}\|Z_1\|\leq C^{\prime}\frac{r}{\alpha_1}, \ \ \|Z_2\|\leq C^{\prime}\frac{r}{\alpha_2} \ \ \makebox{and} \ \ \|Z_3\|\leq C^{\prime}\big(\frac{r}{\alpha_3}+\frac{r^2}{\alpha_1\alpha_2}\big). \end{equation*}
  
\noindent Therefore, letting $S_r\subset \fr{m}$ be the box

 \[ S_r:=\bigg[-C^{\prime}\frac{r}{\alpha_1},C^{\prime}\frac{r}{\alpha_1}\bigg]^{d_1}\times \bigg[-C^{\prime}\frac{r}{\alpha_2},C^{\prime}\frac{r}{\alpha_2}\bigg]^{d_2}\times \bigg[-C^{\prime}\big(\frac{r}{\alpha_3}+\frac{r^2}{\alpha_1\alpha_2}\big),C^{\prime}\big(\frac{r}{\alpha_3}+\frac{r^2}{\alpha_1\alpha_2}\big)\bigg]^{d_3}, \] 

\noindent we obtain $Z\in S_r$, and hence $p=\pi(\exp{Z})=\Psi(Z)\in \Psi(S_r)$. Since $p\in B_{\alpha_I}(o,r)$ is arbitrary, we conclude that 

\[ B_{\alpha_I}(o,r)\subseteq \Psi(S_r).\] 

Now choose a closed ball $B$ with respect to $\left.Q\right|_{\fr{m}\times \fr{m}}$, centered at $0\in \fr{m}$, such that $B\subset V_{\fr{m}}$. Let $C_{J}$ be the maximum of the Jacobian of the coordinate function $\Psi$ on $B$. Since $\frac{r}{\alpha_1}\leq \eta$ uniformly in $\alpha_1\leq \alpha_2\leq \alpha_3$, we have $\frac{r}{\alpha_i}\leq \eta$ and $\frac{r^2}{\alpha_1\alpha_2}\leq \eta^2$, uniformly in $\alpha_1\leq \alpha_2\leq \alpha_3$. Hence we may shrink $\eta$ again so that $S_r$ is contained in $B$, uniformly in $\alpha_1\leq \alpha_2\leq \alpha_3$. We obtain

\begin{eqnarray*} V_{\alpha_I}(r)&=&\mu_0(B_{\alpha_I}(o,r))\leq \mu_0(\Psi(S_r))\leq C_J m(S_r)\\ \\ 
&=&C_J(2C^{\prime})^n(\frac{r}{\alpha_1})^{d_1} (\frac{r}{\alpha_2})^{d_2}(\frac{r}{\alpha_3}+\frac{r^2}{\alpha_1\alpha_2})^{d_3}.\end{eqnarray*}

\noindent Now for $r\geq \frac{\alpha_1\alpha_2}{\alpha_3}$, we have $\frac{r}{\alpha_3}\leq \frac{r^2}{\alpha_1\alpha_2}$, or $\frac{r}{\alpha_3}+\frac{r^2}{\alpha_1\alpha_2}\leq 2\frac{r^2}{\alpha_1\alpha_2}$, in which case the previous inequality yields

\[ V_{\alpha_I}(r)\leq C_4(\alpha_1^{d_1+d_3}\alpha_2^{d_2+d_3})^{-1}r^{n+d_3}, \ \ \frac{\alpha_1\alpha_2}{\alpha_3}\leq r \leq \eta \alpha_1, \]

\noindent with $C_4:=2^{d_3}C_J(2C^{\prime})^n$. Since $Vol_{\alpha_I}(r)=\Lambda_{\alpha_I}V_{\alpha_I}(r)$, multiplying the above inequality by $\Lambda_{\alpha_I}$ yields the desired estimate.\qed

\section{Post sub-Riemannian interval of the radius}\label{PostSubRiemannian}

In this section, we consider the post sub-Riemannian interval for the radius $r$, where $r$ is larger than the smallest metric parameter $\alpha_1$. In that case, the directions in $\fr{m}_1$ no longer determine the volume growth, and the latter is controlled by the directions in $\fr{m}_2\oplus \fr{m}_3$. In turn, the growth is of order $r^{n-d_1}$. For the lower bound, we use the function $F^{\delta}$ from Section \ref{HeisenbergLowerSection}.  For the upper bound, we consider a suitable homogeneous fibration whose fiber has tangent space $\fr{m}_1$.

\subsection{Post sub-Riemannian lower bound}\label{SectionPostSubriemannianLower}
The main result of the subsection is the following.

\begin{prop}\label{PostSubRiemannianLowerProp} Let $(G/K,\alpha_I)$ be a homogeneous space satisfying Assumption \ref{Assumption}.  Then there exists a positive constant $C_5$, depending only on the space $G/K$, such that
   
   \[ Vol_{\alpha_I}(r)\geq C_5\alpha_1^{d_1}\big(\frac{\alpha_3}{\alpha_2}\big)^{d_3}r^{n-d_1}, \ \ \alpha_1\leq r\leq \alpha_2, \] 

\noindent uniformly in $\alpha_1\leq \alpha_2\leq \alpha_3$.   
   \end{prop}

\begin{proof} Let $V$ be the open ball given by Proposition \ref{Vprime}.  Choose $\epsilon\in (0,1)$ sufficiently small so that for $\alpha_1\leq r\leq \alpha_2$, the box 

 \[ S_r:=\Big[-\epsilon,\epsilon\Big]^{d_1}\times \Big[-\epsilon \frac{r}{\alpha_2},\epsilon \frac{r}{\alpha_2}\Big]^{d_2}\times  \Big[-\epsilon^2,\epsilon^2\Big]^{d_3} \]

\noindent is contained in $V$, which is possible since $\frac{r}{\alpha_2}\leq 1$. By further decreasing $\epsilon$ if necessary, we may also assume that
  
 \[ \epsilon < \frac{1}{\sqrt{d_1}+\sqrt{d_2}+Cd_3}, \]

\noindent where $C$ is the constant given by Lemma \ref{distancelemma}. We put 

 \[ \delta:=\frac{r}{\alpha_2}\in (0,1].\]
  
 \noindent For any $(X,Y,Z)\in S_r$, with $X\in \fr{m}_1$, $Y\in \fr{m}_2$, $Z=(z_k)_{k=1}^{d_3}\in \fr{m}_3$, we have

 \[ \|X\|\leq \sqrt{d_1}\epsilon, \ \ \|Y\|\leq \sqrt{d_2}\epsilon \frac{r}{\alpha_2} \ \ \makebox{and}\ \ \sum_{k=1}^{d_3}{\sqrt{|z_k|}}\leq d_3\epsilon. \]

\noindent Then in view of the function $F^{\delta}$ defined in Section \ref{HeisenbergLowerSection}, Lemma \ref{distancelemma} yields

\begin{eqnarray*}d_{\alpha_I}\big(o,F^{\delta}(X,Y,Z)\big)&\leq &\alpha_1\|X\|+\alpha_2\|Y\|+\frac{C}{2}(\alpha_1+\delta \alpha_2)\sum_{k=1}^{d_3}{\sqrt{|z_k|}}\\ \\ 
   &\leq& \epsilon\sqrt{d_1}\alpha_1+\epsilon\sqrt{d_2}r+\frac{Cd_3\epsilon}{2}(\alpha_1+r).\end{eqnarray*}

\noindent Since $\alpha_1\leq r$, the above relation yields

 \[ d_{\alpha_I}\big(o,F^{\delta}(X,Y,Z)\big)\leq \epsilon (\sqrt{d_1}+\sqrt{d_2}+Cd_3)r< r,\]

\noindent and hence $F^{\delta}(S_r)\subset B_{\alpha_I}(o,r)$. Along with Proposition \ref{Vprime}, we obtain

\begin{eqnarray*} V_{\alpha_I}(r)&=&\mu_0(B_{\alpha_I}(o,r))\geq \mu_0(F^{\delta}(S_r))\geq \frac{1}{2}\delta^{d_3}m(S_r)\\ \\ 
&=&\frac{1}{2}\delta^{d_3}(2\epsilon)^{d_1+d_2}(2\epsilon^2)^{d_3}\big(\frac{r}{\alpha_2}\big)^{d_2}. \end{eqnarray*}

\noindent Since $\delta=\frac{r}{\alpha_2}$, the above relation yields 

 \[  V_{\alpha_I}(r)\geq C_5 \alpha_2^{d_1-n}r^{n-d_1}, \]

\noindent with $C_5:=\frac{1}{2}(2\epsilon)^{d_1+d_2}(2\epsilon^2)^{d_3}$. Multiplying by $\Lambda_{\alpha_I}$, we obtain the desired bound for the function $Vol_{\alpha_I}(r)=\Lambda_{\alpha_I}V_{\alpha_I}(r)$.\end{proof}

\subsection{Post sub-Riemannian upper bound}\label{PostSubRiemannianUpper} We will prove the corresponding upper bound.

\begin{prop}\label{postsubRiemannianupperprop} Let $(G/K,\alpha_I)$ be a homogeneous space satisfying Assumption \ref{Assumption}.  Then there exists a positive constant $C_6$, depending only on the space $G/K$, such that
   
   \[ Vol_{\alpha_I}(r)\leq C_6\alpha_1^{d_1}\big(\frac{\alpha_3}{\alpha_2}\big)^{d_3}r^{n-d_1}, \ \ 0< r\leq \alpha_2, \] 

\noindent uniformly in $\alpha_1\leq \alpha_2\leq \alpha_3$.   
   \end{prop}

We first recall the following Fubini formula for Riemannian submersions (see e.g. \cite{Sakai}, Theorem 5.6).

\begin{lemma}\label{FubIN}
Let $\pi:(M,g_M)\rightarrow(N,g_N)$ be a Riemannian submersion. Then for any integrable function $\phi:M\rightarrow\mathbb R$, the restriction $\left.\phi\right|_{\pi^{-1}(y)}$ is integrable for almost every $y\in N$, and we have

\[ \int_M{\phi d\mu_{g_M}}=\int_N{\bigg(\int_{\pi^{-1}(y)}{\left.\phi\right|_{\pi^{-1}(y)}(x)d\mu_{\pi^{-1}(y)}(x)}\bigg)d\mu_{g_N}(y)}, \]

\noindent where $\mu_{\pi^{-1}(y)}$ denotes the Riemannian volume measure induced by $g_M$ on the fiber $\pi^{-1}(y)$.
\end{lemma}

\begin{corol}\label{RiemannianSubVol}Let $\pi:(M,g_M)\rightarrow(N,g_N)$ be a Riemannian submersion between compact Riemannian manifolds, such that all fibers have the same volume $v$. Then, for every measurable $S\subseteq N$, it holds

\[\mu_{g_M}(\pi^{-1}(S))=v\mu_{g_N}(S).\] 
 
 \end{corol}

\begin{proof} Since $M$ is compact, it has finite volume, and thus the function
$\phi:=\mathbf{1}_{\pi^{-1}(S)}$ is integrable on $M$. By Lemma \ref{FubIN}, we obtain

\[\mu_{g_M}(\pi^{-1}(S))=\int_S{\mu_{\pi^{-1}(y)}(\pi^{-1}(y))d\mu_{g_N}(y)}=v\mu_{g_N}(S).\]
\end{proof}

Proposition \ref{postsubRiemannianupperprop} will follow from the more general result below.

\begin{prop}\label{FibrationProp} Let $G$ be a compact connected Lie group and let $K\subset H\subset G$ be closed subgroups with Lie algebras $\fr{k}\subset\fr{h}\subset\fr{g}$ respectively. Fix an $\op{Ad}$-invariant inner product $Q$ on $\fr{g}$ and consider the $Q$-orthogonal reductive decomposition

\[\fr{g}=\fr{k}\oplus\underbrace{\fr{p}\oplus\fr{q}}_{\fr{m}},\]

\noindent where $\fr{h}=\fr{k}\oplus\fr{p}$. Let $g$ be a $G$-invariant Riemannian metric on $G/K$ such that $g(\fr{p},\fr{q})=\{0\}$, and

\begin{equation}\label{assumption}
\left.g\right|_{\fr{q}\times\fr{q}}\geq\alpha^2\left.Q\right|_{\fr{q}\times\fr{q}},
\end{equation}

\noindent for some $\alpha>0$. Then there exist positive constants $C$ and $r_0$, depending only on $G$, $K$, $H$ and $Q$, such that

\[\mu_0(B_g(o,r))\leq C\big(\frac{r}{\alpha}\big)^{\dim G/H}, \ \ 0<r\leq\alpha r_0.\]

\noindent Here $\mu_0$ denotes the Riemannian volume measure induced by $\left.Q\right|_{\fr{m}\times\fr{m}}$ on $G/K$.
\end{prop}

\begin{proof} Consider the homogeneous fibration $H/K\rightarrow G/K\xrightarrow{f}G/H$, where $f:G/K\rightarrow G/H$ is the projection $xK\mapsto xH$, $x\in G$. Observe that $\fr{m}=T_{o}(G/K)$, $\fr{q}=T_{eH}(G/H)$ and $\fr{p}=T_{eK}(H/K)$. Denote by $Q_{\fr{m}}$ the $G$-invariant metric on $G/K$ induced by $\left.Q\right|_{\fr{m}\times\fr{m}}$. Since $\fr{q}$ and $\fr{p}$ are $\op{Ad}_H$-invariant and $\op{Ad}_K$-invariant respectively, the restrictions of $Q$ on the subspaces $\fr{q}$ and $\fr{p}$ define a $G$-invariant metric on the base $G/H$ and an $H$-invariant metric on the fiber $H/K$ respectively. We denote those metrics by $Q_{\fr{q}}$ and $Q_{\fr{p}}$, respectively.

\noindent The map $f:(G/K,Q_{\fr{m}})\rightarrow(G/H,Q_{\fr{q}})$ is a Riemannian submersion. The vertical and horizontal distributions $\mathcal{V}$ and $\mathcal{H}$ of $f$ are the $G$-invariant distributions induced by $\fr{p}$ and $\fr{q}$ respectively. Then the hypothesis $g(\fr{p},\fr{q})=\{0\}$ and assumption \eqref{assumption} can be restated as

\begin{equation}\label{assumption2}
g(\mathcal{V},\mathcal{H})=\{0\} \ \ \makebox{and} \ \ g(X,X)\geq\alpha^2Q_{\fr{m}}(X,X), \ \ X\in\mathcal{H}.
\end{equation}

\noindent Now choose $x\in B_g(o,r)$. Given that the space $(G/K,g)$ is complete, we can choose a minimizing (constant speed) geodesic $\gamma:[0,1]\rightarrow G/K$, joining $o$ with $x$. For each $t\in[0,1]$, decompose $\dot{\gamma}(t)=h(t)+v(t)$ into horizontal and vertical components, where $h(t)\in\mathcal{H}_{\gamma(t)}$ and $v(t)\in\mathcal{V}_{\gamma(t)}$. In view of relation \eqref{assumption2}, we have

\[g(\dot{\gamma}(t),\dot{\gamma}(t))\geq g(h(t),h(t))\geq\alpha^2Q_{\fr{m}}(h(t),h(t)). \]

\noindent Since $f$ is a Riemannian submersion, $(df)_{\gamma(t)}$ defines an isometry from $\mathcal{H}_{\gamma(t)}$ to $T_{f(\gamma(t))}(G/H)$, and thus

\[ Q_{\fr{q}}\bigg(\frac{d}{dt}(f\circ\gamma)(t),\frac{d}{dt}(f\circ\gamma)(t)\bigg)
=Q_{\fr{m}}(h(t),h(t)).\]

\noindent By the above two relations, we obtain

\[ \op{Length}_{Q_{\fr{q}}}(f\circ\gamma)\leq\frac{1}{\alpha}\op{Length}_g(\gamma)<\frac{r}{\alpha}. \]

\noindent Since $(f\circ\gamma)(0)=eH$ and $(f\circ\gamma)(1)=f(x)$, it follows that $d_{Q_{\fr{q}}}(f(x),eH)<\frac{r}{\alpha}$, and thus

\begin{equation}\label{Sat} B_g(o,r)\subseteq f^{-1}\bigg(B_{Q_{\fr{q}}}\big(eH,\frac{r}{\alpha}\big)\bigg).\end{equation}

\noindent Since the metric $Q_{\fr{m}}$ on $G/K$ is $G$-invariant, left translations by elements of $G$ map the fiber $H/K=f^{-1}(eH)$ isometrically onto the other fibers of $f$. Thus all fibers are isometric to $(H/K,Q_{\fr{p}})$, and hence have the same volume $\mu_{Q_{\fr{p}}}(H/K)$. Then Corollary \ref{RiemannianSubVol} yields

\begin{equation}\label{Sat1}\mu_0(f^{-1}(S))=\mu_{Q_{\fr{p}}}(H/K)\mu_{Q_{\fr{q}}}(S),\end{equation}

\noindent for every measurable $S\subseteq G/H$. Moreover, since the volume of a Riemannian ball in $(G/H,Q_{\fr{q}})$ is comparable to the volume of a Euclidean ball for sufficiently small $r$, there exist positive constants $C^{\prime}$ and $r_0$, depending only on $G$, $K$, $H$ and $Q$, such that

\begin{equation}\label{Sat2} \mu_{Q_{\fr{q}}}(B_{Q_{\fr{q}}}(eH,r))\leq C^{\prime}r^{\dim G/H}, \ \ 0<r\leq r_0. \end{equation}

\noindent Hence for $0<r\leq\alpha r_0$, relations \eqref{Sat}-\eqref{Sat2} yield

\begin{eqnarray*}
\mu_0(B_g(o,r))&\leq&\mu_0\bigg(f^{-1}\bigg(B_{Q_{\fr{q}}}\big(eH,\frac{r}{\alpha}\big)\bigg)\bigg)=\mu_{Q_{\fr{p}}}(H/K)\mu_{Q_{\fr{q}}}\bigg(B_{Q_{\fr{q}}}\big(eH,\frac{r}{\alpha}\big)\bigg)\\
&\leq&\mu_{Q_{\fr{p}}}(H/K)C^{\prime}\big(\frac{r}{\alpha}\big)^{\dim G/H}.
\end{eqnarray*}

\noindent The desired result follows by setting $C:=\mu_{Q_{\fr{p}}}(H/K)C^{\prime}$.
\end{proof}

We proceed to prove Proposition \ref{postsubRiemannianupperprop}.\\

\noindent \emph{Proof of Proposition \ref{postsubRiemannianupperprop}.} In view of Proposition \ref{FibrationProp}, we let $Q$ be the fixed $\op{Ad}$-invariant inner product on $\fr{g}$, and we set $\fr{h}:=\fr{k}\oplus \fr{m}_1$, 
$\fr{p}:=\fr{m}_1$, $\fr{q}:=\fr{m}_2\oplus \fr{m}_3$
and $g:=\alpha_I$.  By part 2. of Lemma \ref{LemmaRelaxed}, the group $H:=N_G(\fr{h})$ is a closed Lie subgroup of $G$ with Lie algebra $\fr{h}$, such that $K\subset H\subset G$. Consider the decomposition 
    
    \[ \fr{g}=\fr{k}\oplus \underbrace{\fr{p}\oplus \fr{q}}_{\fr{m}}.\]
\noindent We have $g(\fr{p},\fr{q})=\{0\}$ and, since $\alpha_2\leq \alpha_3$, it holds $g(X,X)\geq \alpha_2^2Q(X,X)$ for all $X\in \fr{q}$. By Proposition \ref{FibrationProp}, there exist positive constants $C$ and $r_0$, depending only on $G$, $K$, $H$ and $Q$, such that 

    \[ V_{\alpha_I}(r)\leq C\big(\frac{r}{\alpha_2}\big)^{n-d_1}, \ \ 0<r\leq \alpha_2r_0. \]

  \noindent Now if $1\leq r_0$, the proof is complete after multiplying the above inequality by $\Lambda_{\alpha_I}$. Otherwise, for $\alpha_2r_0\leq r\leq \alpha_2$, we have

 \[ V_{\alpha_I}(r)\leq \mu_0(G/K)\leq \frac{\mu_0(G/K)}{r_0^{n-d_1}}\big(\frac{r}{\alpha_2}\big)^{n-d_1}, \]

 \noindent where we have used the fact that $G$ is compact so that $\mu_0(G/K)<\infty$.  Setting 
 
 \[ C_6:=\max\bigg\{C,\frac{\mu_0(G/K)}{r_0^{n-d_1}}\bigg\}, \]
 
 \noindent and multiplying the above inequality by $\Lambda_{\alpha_I}$, yields the desired estimate for $0<r\leq \alpha_2$. Observe that under Assumption \ref{Assumption}, the constants $C$ and $r_0$, and thus the constant $C_6$, depend only on the space $G/K$ and the fixed ordering of the subspaces $\fr{m}_i$. In particular, they do not depend on the metric $\alpha_I$. \qed

\section{Proof of theorems \ref{maintheorem} and \ref{VolumeGrowth}}\label{MainProofsSection}

First, we need the following.

\begin{lemma}\label{FixBound}Let $(G/K,\alpha_I)$ be a homogeneous space satisfying Assumption \ref{Assumption}.  Then there exists a constant $C_7$, depending only on the space $G/K$, such that 
  \[ Vol_{\alpha_I}(r)\leq C_7\big(\frac{\alpha_3}{\alpha_1\alpha_2}\big)^{d_3}r^{n+d_3}, \ \ \frac{\alpha_1\alpha_2}{\alpha_3}\leq r\leq \alpha_1,\]
 
 \noindent uniformly in $\alpha_1\leq \alpha_2\leq \alpha_3$.\end{lemma}

\begin{proof} By Proposition \ref{UpperHeisenbergProp}, we have the estimate 

\[ Vol_{\alpha_I}(r)\leq C_4\big(\frac{\alpha_3}{\alpha_1\alpha_2}\big)^{d_3}r^{n+d_3}, \ \ \frac{\alpha_1\alpha_2}{\alpha_3}\leq r \leq \eta\alpha_1. \]

\noindent  Assume that $\eta \alpha_1\leq r\leq \alpha_1$. By Proposition \ref{postsubRiemannianupperprop}, we have

\[Vol_{\alpha_I}(r)\leq C_6\alpha_1^{d_1}\left(\frac{\alpha_3}{\alpha_2}\right)^{d_3}r^{n-d_1}
=C_6\left(\frac{\alpha_1}{r}\right)^{d_1+d_3}\left(\frac{\alpha_3}{\alpha_1\alpha_2}\right)^{d_3}r^{n+d_3}, \ \ 0<r\leq \alpha_2.
\]

\noindent Since $r\geq \eta\alpha_1$, we have $(\frac{\alpha_1}{r})^{d_1+d_3}\leq \eta^{-(d_1+d_3)}$, and thus

\[Vol_{\alpha_I}(r)\leq C_6\eta^{-(d_1+d_3)}\left(\frac{\alpha_3}{\alpha_1\alpha_2}\right)^{d_3}r^{n+d_3}, \ \ \eta\alpha_1\leq r\leq \alpha_1,\]

\noindent which yields the desired estimate for $C_7:=\max\{C_4,C_6\eta^{-(d_1+d_3)}\}$.
\end{proof}

Having established the bounds in the corresponding radius intervals, we are ready to prove Theorem \ref{VolumeGrowth}.\\

\noindent \emph{Proof of Theorem \ref{VolumeGrowth}.} For the interval $0<r\leq \frac{\alpha_1\alpha_2}{\alpha_3}$, Corollary  \ref{CorolEuclideanFinal} yields

\[ C_1r^n\leq Vol_{\alpha_I}(r)\leq C_2r^n. \]

\noindent Thus the definition of $\widetilde{Vol}_{\alpha_I}(r)$ yields 

\begin{equation}\label{Growth1}C_1\leq \frac{Vol_{\alpha_I}(r)}{\widetilde{Vol}_{\alpha_I}(r)}\leq C_2, \ \ 0<r\leq \frac{\alpha_1\alpha_2}{\alpha_3}. \end{equation}

\noindent For the interval $\frac{\alpha_1\alpha_2}{\alpha_3}\leq r \leq \alpha_1$, Proposition \ref{HeisenbergLower}, Lemma \ref{FixBound} and the definition of $\widetilde{Vol}_{\alpha_I}(r)$ give 

\begin{equation}\label{Kioski1} C_3\leq \frac{Vol_{\alpha_I}(r)}{\widetilde{Vol}_{\alpha_I}(r)}\leq C_7, \ \ \frac{\alpha_1\alpha_2}{\alpha_3}\leq r\leq \alpha_1.\end{equation}

 \noindent Similarly, for the interval $\alpha_1\leq r\leq \alpha_2$, Propositions \ref{PostSubRiemannianLowerProp}, \ref{postsubRiemannianupperprop} and the definition of $\widetilde{Vol}_{\alpha_I}(r)$ yield

\begin{equation}\label{Kioski2}  C_5\leq \frac{Vol_{\alpha_I}(r)}{\widetilde{Vol}_{\alpha_I}(r)}\leq C_6, \ \ \alpha_1\leq r \leq \alpha_2.\end{equation}

\noindent Finally, for $r\geq\alpha_2$, the monotonicity of $Vol_{\alpha_I}(r)$, along with Proposition \ref{PostSubRiemannianLowerProp}, yield

\[Vol_{\alpha_I}(r)\geq Vol_{\alpha_I}(\alpha_2)\geq C_5\alpha_1^{d_1}\alpha_2^{d_2}\alpha_3^{d_3}. \]

\noindent We also have the trivial upper bound
\[ Vol_{\alpha_I}(r)\leq \alpha_1^{d_1}\alpha_2^{d_2}\alpha_3^{d_3}\mu_0(G/K). \]

\noindent Therefore, the definition of $\widetilde{Vol}_{\alpha_I}(r)$ yields 

\begin{equation}\label{Growth4}  C_5\leq \frac{Vol_{\alpha_I}(r)}{\widetilde{Vol}_{\alpha_I}(r)}
\leq \mu_0(G/K), \ \ r\geq \alpha_2.  \end{equation}

\noindent Now taking 

\begin{eqnarray*} b_1&:=&\min\{C_1, C_3, C_5\},\\  \\ 
b_2&:=&\max\{C_2, C_7, C_6, \mu_0(G/K)\},\end{eqnarray*} 

\noindent along with relations \eqref{Growth1} - \eqref{Growth4}, we obtain 

\[ b_1\leq \frac{Vol_{\alpha_I}(r)}{\widetilde{Vol}_{\alpha_I}(r)}\leq b_2. \]

\noindent Observe that the constants $b_1,b_2$ depend only on the space $G/K$, under Assumption \ref{Assumption}, and on the ordering of the spaces $\fr{m}_i$, $i=1,2,3$. But there are $6$ such possible orderings, so taking the minimum of $b_1$ and the maximum of $b_2$ over all $6$ orderings, we obtain the desired result. \qed  \\

 We are ready to prove Theorem \ref{maintheorem}. \\

 \noindent \emph{Proof of Theorem \ref{maintheorem}.}  Set $d_{max}:=\max\{d_1,d_2,d_3\}$. We will show that 
 
 \begin{equation}\label{TwotoC}\widetilde{Vol}_{\alpha_I}(2r)\leq 2^{n+d_{max}}\widetilde{Vol}_{\alpha_I}(r), \ \ r>0.\end{equation}

\noindent To this end, consider the function

\[ f(r):=\frac{\widetilde{Vol}_{\alpha_I}(r)}{r^{n+d_{max}}}=\begin{cases} r^{-d_{max}},  & 0<r\leq \frac{\alpha_1\alpha_2}{\alpha_3},\\ 
\big(\frac{\alpha_3}{\alpha_1\alpha_2}\big)^{d_3}r^{d_3-d_{max}},  & \frac{\alpha_1\alpha_2}{\alpha_3}<r\leq \alpha_1, \\  
\alpha_1^{d_1}(\frac{\alpha_3}{\alpha_2})^{d_3}r^{-d_1-d_{max}}, &  \alpha_1<r\leq \alpha_2, \\ 
\alpha_1^{d_1}\alpha_2^{d_2}\alpha_3^{d_3}r^{-n-d_{max}}, & r>\alpha_2  \end{cases}. \] 
 
 \noindent Then $f$ is continuous since $\widetilde{Vol}_{\alpha_I}(r)$ is continuous. Besides, $f(r)$ is non-increasing on $(0,+\infty)$, and hence $f(2r)\leq f(r)$, which yields \eqref{TwotoC}. Now from relation \eqref{TwotoC} and Theorem \ref{VolumeGrowth}, we obtain

 \begin{equation*}Vol_{\alpha_I}(2r)\leq b_2\widetilde{Vol}_{\alpha_I}(2r)\leq b_2 2^{n+d_{max}}\widetilde{Vol}_{\alpha_I}(r)\leq 2^{n+d_{max}}\frac{b_2}{b_1}Vol_{\alpha_I}(r). \end{equation*}

 \noindent We conclude that $D_{\alpha_I}\leq 2^{n+d_{max}}\frac{b_2}{b_1}$, for all metrics $\alpha_I$. Since the right hand depends only on the space $G/K$, the proof is concluded. \qed

\section{Applications of Theorem \ref{maintheorem}}\label{Z2WallachApplications}

In this section, we discuss the classes of $\mathbb Z_2\times \mathbb Z_2$-symmetric spaces and generalized Wallach spaces. We also prove Theorem \ref{TheoremWallach}. 

Let $G/K$ be a homogeneous space where $G$ is a compact connected Lie group, acting almost effectively on the space. Following \cite{Kol09}, the space $G/K$ is called a $Z_2\times \mathbb Z_2$-symmetric space if there exists an injective homomorphism $\rho:\mathbb Z_2\times \mathbb Z_2\rightarrow \op{Aut}(G)$, such that 
 
 \[ G^{\rho}_0\subseteq K \subseteq G^{\rho},\]

 \noindent where $G^{\rho}$ denotes the subgroup of fixed points of $\rho(Z_2\times \mathbb Z_2)$, and $G^{\rho}_0$ denotes its connected component at $e$. Equivalently, there exist two different commuting involutions $\sigma,\tau\in \op{Aut}(G)\setminus \{\op{id}\}$, such that 

 \[ (G^{\sigma}\cap G^{\tau})_0\subseteq K\subseteq 
G^{\sigma}\cap G^{\tau}.\]

\noindent Let $\fr{g}=\fr{g}^{\sigma}\oplus\fr{p}^{\sigma}=\fr{g}^{\tau}\oplus\fr{p}^{\tau}$ be the corresponding decompositions of $\fr{g}$ into $+1$ and $-1$ eigenspaces of $\sigma, \tau$ respectively. Consider the spaces 

 \[ \fr{k}=\op{Lie}(K)=\fr{g}^{\sigma}\cap \fr{g}^{\tau}, \ \fr{m}_1:=\fr{g}^{\sigma}\cap \fr{p}^{\tau}, \ \fr{m}_2:=\fr{g}^{\tau}\cap \fr{p}^{\sigma}, \ \fr{m}_3:=\fr{p}^{\sigma}\cap \fr{p}^{\tau}. \]

\noindent The spaces $\fr{m}_1,\fr{m}_2,\fr{m}_3$ are $\op{Ad}_K$-invariant. Moreover, by \cite{Kol09}, Proposition 2.2, we have the decomposition

\begin{equation}\label{KollrossDecomposition}\fr{g}=\fr{k}\oplus \fr{m}_1\oplus \fr{m}_2\oplus \fr{m}_3,\end{equation}

\noindent and the corresponding $\mathbb Z_2\times\mathbb Z_2$-grading yields

\begin{equation}\label{KolrosBracket}[\fr{m}_i,\fr{m}_i]\subseteq \fr{k}, \ \ i=1,2,3, \ \ \makebox{and} \ \ [\fr{m}_i,\fr{m}_j]\subseteq \fr{m}_k, \ \ \makebox{$i,j,k$ pairwise distinct}. \end{equation}

\noindent  The following lemma ensures that the inclusions $[\fr{m}_i,\fr{m}_j]\subseteq \fr{m}_k$ are in fact equalities for $\fr{g}$ simple.

\begin{lemma}\label{Z2Satisfy}Let $G/K$ be a $\mathbb Z_2\times \mathbb Z_2$-symmetric space. If $\fr{g}$ is simple, then
 
\begin{equation*}[\fr{m}_i,\fr{m}_j]=\fr{m}_k, \ \  \makebox{for all} \ \ i,j,k \ \ \makebox{pairwise distinct}. \end{equation*}

\end{lemma}

\begin{proof} For any $i,j=1,2,3$, with $i\neq j$, consider the space $\fr{s}:=\fr{m}_i\oplus \fr{m}_j$. We will show that the space $\fr{s}+[\fr{s},\fr{s}]$ is an ideal of $\fr{g}$. Indeed, by the $\op{Ad}_K$-invariance of the submodules $\fr{m}_l$, we have 
\[
[\fr{k},\fr{s}]\subseteq \fr{s},
\]
\noindent while the Jacobi identity yields $[\fr{k},[\fr{s},\fr{s}]]\subseteq[\fr{s},\fr{s}]$. Moreover, relations \eqref{KolrosBracket} give
\[
[\fr{s},\fr{s}+[\fr{s},\fr{s}]]\subseteq \fr{s}+[\fr{s},\fr{s}] \ \ \makebox{and} \ \
[\fr{m}_k,\fr{s}]\subseteq\fr{s},
\]
\noindent where $k\neq i,j$. Hence, by the Jacobi identity we obtain
\[
[\fr{m}_k,[\fr{s},\fr{s}]]\subseteq[\fr{s},\fr{s}].
\]
\noindent Given that $\fr{g}=\fr{k}\oplus \fr{m}_k\oplus \fr{s}$, the above relations show that $\fr{s}+[\fr{s},\fr{s}]$ is an ideal of $\fr{g}$. Notice that $\fr{s}\neq\{0\}$ for otherwise, either $\sigma=\tau$ or at least one of the involutions $\sigma,\tau$ is the identity, contradicting the injectivity of $\rho$. Since $\fr{g}$ is simple and $\fr{s}\neq \{0\}$, we deduce that 
 
 \[\fr{g}= \fr{s}+[\fr{s},\fr{s}]
=\fr{m}_i\oplus \fr{m}_j+[\fr{m}_i,\fr{m}_i]+[\fr{m}_j,\fr{m}_j]+[\fr{m}_i,\fr{m}_j]\subseteq
\fr{k}\oplus \fr{m}_i\oplus \fr{m}_j\oplus[\fr{m}_i,\fr{m}_j].\]

\noindent Along with relations \eqref{KolrosBracket} and decomposition \eqref{KollrossDecomposition}, the above relation yields $[\fr{m}_i,\fr{m}_j]=\fr{m}_k$, for $k\neq i,j$.\end{proof}

 A generalized Wallach space is defined as an almost effective homogeneous space $G/K$ with $G$ compact connected and semisimple, such that the isotropy representation $\left.\op{Ad}_K\right|_{\fr{m}}:K\rightarrow \op{Gl}(\fr{m})$ decomposes into three pairwise orthogonal, irreducible summands $\fr{m}=\fr{m}_1\oplus \fr{m}_2\oplus\fr{m}_3$, with respect to the Killing form $B$ of $\fr{g}$, satisfying

\begin{equation*}[\fr{m}_i,\fr{m}_i]\subseteq \fr{k}, \ \ i=1,2,3. \end{equation*}

\noindent By the $\op{Ad}$-invariance of the Killing form and the $B$-orthogonality of the above decomposition, the above relations imply

\begin{equation*} [\fr{m}_i,\fr{m}_j]\subseteq \fr{m}_k, \ \ \makebox{for all $i,j,k$ pairwise distinct}.\end{equation*}

\noindent Since the submodules $\fr{m}_i$ are $\op{Ad}_K$-irreducible, part 3. of Lemma \ref{LemmaRelaxed} yields the following.

\begin{lemma}\label{LemmaRelaxed1} Let $G/K$ be a generalized Wallach space.  Then either 

\begin{eqnarray*}[\fr{m}_i,\fr{m}_j]&=&\{0\}, \ \ \makebox{for all $i\neq j$,  or} \\ \\  
{}[\fr{m}_i,\fr{m}_j]&=&\fr{m}_k, \ \ \makebox{for all $i,j,k$ pairwise distinct}. \end{eqnarray*}
\end{lemma}

\begin{remark}\label{ProductLocal}By the results in \cite{Nik16}, the first case in Lemma \ref{LemmaRelaxed1} holds true if and only if $G/K$ is locally the product of three compact irreducible symmetric spaces.\end{remark}

Generalized Wallach spaces include the standard Wallach spaces $SU(3)/T_{max}$, $Sp(3)/Sp(1)^3$, $F_4/Spin(8)$ and the Lie group $SU(2)$, among several other homogeneous classes. They have been independently classified in \cite{Nik16} (correction in \cite{Nik21}) and in \cite{CheKaLi16}.  We remark that any generalized Wallach space $G/K$ with $G$ simple is a $\mathbb Z_2\times \mathbb Z_2$-symmetric space (\cite{Nik16}).  From the aforementioned classifications, we deduce that the generalized Wallach spaces listed in Table \ref{table} have non-equivalent submodules among $\fr{m}_i$.  We proceed to prove Theorem \ref{TheoremWallach}.\\

\noindent \emph{Proof of Theorem \ref{TheoremWallach}.} 
 Suppose that $G/K$ is a $\mathbb Z_2\times \mathbb Z_2$-symmetric space with $G$ compact and simple. Let $Q$ be the negative of the Killing form of $\fr{g}$. Then $\fr{g}$ admits the $Q$-orthogonal decomposition \eqref{KollrossDecomposition}, satisfying relation \eqref{KolrosBracket}. By Lemma \ref{Z2Satisfy}, relations \eqref{cyclicbracket} are satisfied, and hence $G/K$ satisfies Assumption \ref{Assumption}.  By Theorem \ref{maintheorem}, $G/K$ is uniformly doubling on the set of metrics $\alpha_I$ of the form \eqref{MetricParameterForm}. If $G/K$ is a generalized Wallach space that is not locally a product of three irreducible compact symmetric spaces, Remark \ref{ProductLocal}, along with Lemma \ref{LemmaRelaxed1}, imply that $G/K$ satisfies Assumption \ref{Assumption}. By Theorem \ref{maintheorem}, $G/K$ is uniformly doubling on the set of metrics $\alpha_I$. On the other hand, if $G/K$ is locally the product of three symmetric spaces, Remark \ref{ProductLocal} implies that $[\fr{m}_i,\fr{m}_j]=\{0\}$, for all $i\neq j$. By Remark \ref{UniformDoublingProduct}, the space is uniformly doubling with $D_{\alpha_I}\leq 2^{n}$ for all $G$-invariant metrics $\alpha_I$. This settles the first part of the theorem.

Now let $G/K$ be one of the generalized Wallach spaces in Table \ref{table} and thus the submodules $\fr{m}_i$ are inequivalent. Then Schur's lemma implies that the metrics $\alpha_I$ exhaust the $G$-invariant metrics on $G/K$. Therefore, $G/K$ is uniformly doubling on the complete set of $G$-invariant metrics.  If $G/K=SU(2)$, although the submodules $\fr{m}_i=<e_i>$ are equivalent, any left-invariant metric can be diagonalized with respect to a Milnor basis (see e.g. \cite{EGS18}), and hence has the form \eqref{MetricParameterForm}. The same is true for the Stiefel manifold $V_2\mathbb R^n=SO(n)/SO(n-2)$, $n\geq 3$, $n\neq 4$. Indeed, if $n=3$, then $V_2\mathbb R^3=SO(3)$, and the diagonalization follows from the analogous result for $SU(2)$ via the double covering $SU(2)\rightarrow SO(3)$. If $n\geq 5$, athough two of the submodules $\fr{m}_{i}$ are equivalent, any metric can be diagonalized as \eqref{MetricParameterForm} via the adjoint action of $SO(2)$ (\cite{Kerr} p. 120, 121). Therefore, the metrics $\alpha_I$ exhaust the $G$-invariant metrics in those spaces up to isometry. We conclude that $SU(2)$ and $SO(n)/SO(n-2)$, $n\geq 3$, $n\neq 4$, are uniformly doubling on the complete set of $G$-invariant metrics.\qed

\section{Proof of Theorem \ref{PoincareTheorem}}\label{PoincareProofSection}

In this section, we prove Theorem \ref{PoincareTheorem}, obtaining the global Poincar\'{e} inequality \eqref{PoincareInequality} for compact homogeneous spaces $G/K$. Corollary \ref{PoincareCorol} follows immediately from theorems \ref{PoincareTheorem} and \ref{TheoremWallach}, by taking $D$ to be the corresponding uniform doubling constant and observing that the isotropy groups $K$ in the spaces $G/K$ of Table \ref{table} are connected.\\

\noindent \emph{Proof of Theorem \ref{PoincareTheorem}.} By homogeneity, it suffices to prove the inequality for balls centered at $o$. Fix an $\op{Ad}$-invariant inner product $Q$ on $\fr{g}$ and consider the $Q$-orthogonal reductive decomposition
   
   \[ \fr{g}=\fr{k}\oplus \fr{m}. \]
\noindent The $G$-invariant metric $g$ corresponds to an $\op{Ad}_K$-invariant inner product on $\fr{m}$, which we denote again by $g$. For $\epsilon>0$, we consider the left-invariant metric $\bar{g}_{\epsilon}$ on $G$, induced by the inner product    
   
\[ \bar{g}_{\epsilon}=\left.g\right|_{\fr{m}\times \fr{m}}+\left.\epsilon^2Q\right|_{\fr{k}\times \fr{k}}. \]

\noindent Observe that $\bar{g}_{\epsilon}$ is $\op{Ad}_K$-invariant, and thus the projection $\pi:(G,\bar{g}_{\epsilon})\rightarrow (G/K,g)$ is a Riemannian submersion. Every fiber is isometric to $(K,\left.\epsilon^2Q\right|_{\fr{k}\times \fr{k}})$, and hence all fibers have the same Riemannian volume. Set

\[ v_{\epsilon}:=\mu_{\pi^{-1}(p)}(\pi^{-1}(p)), \ \ p\in G/K, \]

\noindent where $\mu_{\pi^{-1}(p)}$ denotes the Riemannian volume measure induced by $\bar{g}_{\epsilon}$ on the fiber $\pi^{-1}(p)$. Corollary \ref{RiemannianSubVol} yields

\begin{equation}\label{above} \mu_{\bar{g}_{\epsilon}}(\pi^{-1}(S))=v_{\epsilon}\mu_{g}(S), \end{equation}

\noindent for every measurable set $S\subseteq G/K$.
Now let $f\in \mathcal{C}^{\infty}(G/K)$ and set $F:=f\circ \pi\in \mathcal{C}^{\infty}(G)$. By the proof of Poincar\'{e} inequality for unimodular Lie groups in \cite{SC02}, p.172, 173, we obtain

\begin{equation}\label{PoincareLieGroupsProof} \int_{B_{\bar{g}_{\epsilon}}(e,r)}{|F-F_{B_{\bar{g}_{\epsilon}}(e,r)}|^2d\mu_{\bar{g}_{\epsilon}}}\leq 2r^2\frac{\mu_{\bar{g}_{\epsilon}}(B_{\bar{g}_{\epsilon}}(e,2r))}{\mu_{\bar{g}_{\epsilon}}(B_{\bar{g}_{\epsilon}}(e,r))}\int_{B_{\bar{g}_{\epsilon}}(e,2r)}{\|\nabla_{\bar{g}_{\epsilon}}{F}\|_{\bar{g}_{\epsilon}}^2d\mu_{\bar{g}_{\epsilon}}}, \end{equation}
  
\noindent where

\[ F_{B_{\bar{g}_{\epsilon}}(e,r)}=\frac{1}{\mu_{\bar{g}_{\epsilon}}(B_{\bar{g}_{\epsilon}}(e,r))}\int_{B_{\bar{g}_{\epsilon}}(e,r)}{Fd\mu_{\bar{g}_{\epsilon}}}.\]

\noindent Since $G$ is compact, it is unimodular and thus the above inequality holds for $G$.  To prove the corresponding inequality for $G/K$, we will first prove the following three claims.
 
 \begin{claim}\label{ClaimPoincare1} \begin{equation*}\lim_{\epsilon \rightarrow 0}{\frac{\mu_{\bar{g}_{\epsilon}}(B_{\bar{g}_{\epsilon}}(e,2r))}{\mu_{\bar{g}_{\epsilon}}(B_{\bar{g}_{\epsilon}}(e,r))}}=\frac{\mu_g(B_g(o,2r))}{\mu_g(B_g(o,r))}\leq D_g.\end{equation*} \end{claim}   
 
 \begin{claim}\label{ClaimPoincare2} \begin{equation*} \lim_{\epsilon\rightarrow 0}{\frac{1}{v_{\epsilon}}\int_{B_{\bar{g}_{\epsilon}}(e,2r)}{\|\nabla_{\bar{g}_{\epsilon}}{F}\|_{\bar{g}_{\epsilon}}^2d\mu_{\bar{g}_{\epsilon}}}}=\int_{B_g(o,2r)}{\|\nabla_gf\|^2_gd\mu_g}.\end{equation*}\end{claim}

\begin{claim}\label{ClaimPoincare3} \begin{equation*} \lim_{\epsilon\rightarrow 0}{\frac{1}{v_{\epsilon}}\int_{B_{\bar{g}_{\epsilon}}(e,r)}{|F-F_{B_{\bar{g}_{\epsilon}}(e,r)}|^2d\mu_{\bar{g}_{\epsilon}}}}=\int_{B_g(o,r)}{|f-f_{B_g(o,r)}|^2d\mu_g}.\end{equation*}\end{claim}

\noindent If the three claims are proven, then dividing both sides of \eqref{PoincareLieGroupsProof} by $v_{\epsilon}$ and taking the limit $\epsilon\rightarrow 0$, yields the desired result.

For the first claim, we set 

\[ \lambda_{\epsilon}=\op{diam}(K,\left.\epsilon^2Q\right|_{\fr{k}\times \fr{k}})=\epsilon \op{diam}(K,\left.Q\right|_{\fr{k}\times \fr{k}}). \]

\noindent Since $K$ is connected and compact, its diameter is finite and thus $\lim_{\epsilon\rightarrow 0}{\lambda_{\epsilon}}=0$. Let $\epsilon$ be sufficiently small so that $\lambda_{\epsilon}<r$. First, we show that

\begin{equation}\label{ClaimBall}\pi^{-1}(B_g(o,r-\lambda_{\epsilon}))\subseteq B_{\bar{g}_{\epsilon}}(e,r)\subseteq \pi^{-1}(B_g(o,r)).  \end{equation}

\noindent The second inclusion holds since $\pi$ is a Riemannian submersion, and thus for any $x\in B_{\bar{g}_{\epsilon}}(e,r)$, we have $d_g(o,\pi(x))\leq d_{\bar{g}_{\epsilon}}(e,x)< r$. For the first inclusion, choose $x\in \pi^{-1}(B_g(o,r-\lambda_{\epsilon}))$ so that 

\[ d_g(o,\pi(x))<r-\lambda_{\epsilon}.\]

\noindent Now choose a minimizing geodesic $\gamma:[0,1]\rightarrow G/K$ joining $o$ with $\pi(x)$, and consider its unique horizontal lift $\alpha:[0,1]\rightarrow G$ with $\alpha(0)=e$, which exists globally since $(G,\bar{g}_{\epsilon})$ is complete.  Then $\pi(\alpha(1))=\pi(x)$, i.e., both $x$ and $\alpha(1)$ belong on the same fiber $xK$, and hence

\begin{equation}\label{SameFiber}d_{\bar{g}_{\epsilon}}(x,\alpha(1))\leq \op{diam}(xK,\left.\epsilon^2Q\right|_{\fr{k}\times \fr{k}})= \lambda_{\epsilon}.\end{equation}
 
\noindent On the other hand, since the horizontal lift preserves the length of a curve, we have

\begin{equation}\label{SameFiber1} d_{\bar{g}_{\epsilon}}(e,\alpha(1))\leq \op{Length}_{\bar{g}_{\epsilon}}(\alpha)=\op{Length}_g(\gamma)=d_g(o,\pi(x))<r-\lambda_{\epsilon}.    \end{equation}

\noindent Relations \eqref{SameFiber}, \eqref{SameFiber1} and the triangle inequality yield $d_{\bar{g}_{\epsilon}}(e,x)<r$, giving the first inclusion.

 Now applying Equation \eqref{above} to relation \eqref{ClaimBall}, we obtain

\begin{equation}\label{Mang} \mu_g(B_g(o,r-\lambda_{\epsilon}))\leq \frac{\mu_{\bar{g}_{\epsilon}}(B_{\bar{g}_{\epsilon}}(e,r))}{v_{\epsilon}}\leq \mu_g(B_g(o,r)). \end{equation}

\noindent Since $\lambda_{\epsilon} \downarrow 0$ as $\epsilon \downarrow 0$, we have $B_g(o,r-\lambda_{\epsilon})\uparrow B_g(o,r)$. By continuity from below of $\mu_g$, we obtain 

 \begin{equation}\label{LowContinuity} \lim_{\epsilon\rightarrow 0}{\mu_g(B_g(o,r-\lambda_{\epsilon}))}=\mu_g(B_g(o,r)).\end{equation}

\noindent Therefore, inequality \eqref{Mang} yields

\begin{equation}\label{LimitVolume}\lim_{\epsilon \rightarrow 0}{\frac{\mu_{\bar{g}_{\epsilon}}(B_{\bar{g}_{\epsilon}}(e,r))}{v_{\epsilon}}}=\mu_g(B_g(o,r)).   \end{equation}

   \noindent By also applying relation \eqref{LimitVolume} with $2r$ in place of $r$, we obtain

\begin{equation*}\lim_{\epsilon \rightarrow 0}{\frac{\mu_{\bar{g}_{\epsilon}}(B_{\bar{g}_{\epsilon}}(e,2r))}{\mu_{\bar{g}_{\epsilon}}(B_{\bar{g}_{\epsilon}}(e,r))}}=\frac{\mu_g(B_g(o,2r))}{\mu_g(B_g(o,r))}\leq D_g,\end{equation*}

\noindent thus concluding Claim \ref{ClaimPoincare1}.

  For Claim \ref{ClaimPoincare2}, since $F=f\circ \pi$, the gradient $\nabla_{\bar{g}_{\epsilon}}{F}$ is the horizontal lift of the gradient $\nabla_gf$ (see e.g. \cite{BMP12}), and hence

\begin{equation*}  \|(\nabla_{\bar{g}_{\epsilon}}{F})(x)\|^2_{\bar{g}_{\epsilon}}=\|(\nabla_gf)(\pi(x))\|^2_g, \ \ x\in G.     \end{equation*}

\noindent Set $h:=\|(\nabla_gf)\|^2_g:G/K\rightarrow \mathbb R$. Then $h\circ\pi=\|(\nabla_{\bar{g}_{\epsilon}}{F})\|^2_{\bar{g}_{\epsilon}}:G\rightarrow \mathbb R$. For every measurable set $S\subseteq G/K$, applying Lemma \ref{FubIN} to the integrable function $(h\circ\pi)\mathbf{1}_{\pi^{-1}(S)}$ yields

\begin{eqnarray}
\int_{\pi^{-1}(S)}{\|(\nabla_{\bar{g}_{\epsilon}}{F})\|^2_{\bar{g}_{\epsilon}}d\mu_{\bar{g}_{\epsilon}}}&=&\int_S{\bigg(\int_{\pi^{-1}(p)}{(h\circ\pi)(x)d\mu_{\pi^{-1}(p)}}(x)\bigg)d\mu_g(p)} \nonumber \\
&=&v_{\epsilon}\int_S{hd\mu_g}=v_{\epsilon}\int_{S}{\|\nabla_gf\|^2_gd\mu_g},\label{FubiniGradient1}
\end{eqnarray}

\noindent where in the second equality we have used the fact that $h\circ\pi$ is constant on the fiber $\pi^{-1}(p)$ with value $h(p)$. Relation \eqref{ClaimBall}, along with Equation \eqref{FubiniGradient1}, yields 

\begin{equation}\label{SandwitchIntegral}
\int_{B_g(o,r-\lambda_{\epsilon})}{\|\nabla_gf\|^2_gd\mu_g}\leq \frac{1}{v_{\epsilon}}\int_{B_{\bar{g}_{\epsilon}}(e,r)}{\|\nabla_{\bar{g}_{\epsilon}}{F}\|_{\bar{g}_{\epsilon}}^2d\mu_{\bar{g}_{\epsilon}}}\leq \int_{B_g(o,r)}{\|\nabla_gf\|^2_gd\mu_g}. \end{equation}

\noindent Since $G/K$ is compact and $f$ is smooth, the function $\|\nabla_gf\|^2_g$ is bounded above by a positive constant $C$, and thus 

 \begin{eqnarray*} \bigg|\int_{B_g(o,r)}{\|\nabla_gf\|^2_gd\mu_g}-\int_{B_g(o,r-\lambda_{\epsilon})}{\|\nabla_gf\|^2_gd\mu_g}\bigg|&=&\bigg|\int_{B_g(o,r)\setminus B_g(o,r-\lambda_{\epsilon})}{\|\nabla_gf\|^2_gd\mu_g}\bigg|\\ \\  
 &\leq& C \big|\mu_g(B_g(o,r))-\mu_g(B_g(o,r-\lambda_{\epsilon}))\big|. \end{eqnarray*}

\noindent By relation \eqref{LowContinuity}, the right-hand side tends to zero as $\epsilon \rightarrow 0$. Therefore, 

\[ \lim_{\epsilon\rightarrow 0}{\int_{B_g(o,r-\lambda_{\epsilon})}{\|\nabla_gf\|^2_gd\mu_g}}=\int_{B_g(o,r)}{\|\nabla_gf\|^2_gd\mu_g}, \]

\noindent and thus relation \eqref{SandwitchIntegral} yields

\begin{equation*} \lim_{\epsilon\rightarrow 0}{\frac{1}{v_{\epsilon}}\int_{B_{\bar{g}_{\epsilon}}(e,r)}{\|\nabla_{\bar{g}_{\epsilon}}{F}\|_{\bar{g}_{\epsilon}}^2d\mu_{\bar{g}_{\epsilon}}}}=\int_{B_g(o,r)}{\|\nabla_gf\|^2_gd\mu_g}.\end{equation*}

\noindent Applying the above relation with $2r$ in place of $r$, proves Claim \ref{ClaimPoincare2}.

For Claim \ref{ClaimPoincare3}, we will firstly show that 

\begin{equation}\label{FirstLimitMean}\lim_{\epsilon \rightarrow 0}{F_{B_{\bar{g}_{\epsilon}}(e,r)}}=f_{B_g(o,r)}. \end{equation}

\noindent In view of relation \eqref{LimitVolume} and the definition of the mean, it suffices to prove that

 \begin{equation}\label{LimitMean1}\lim_{\epsilon \rightarrow 0}{\frac{1}{v_{\epsilon}}\int_{B_{\bar{g}_{\epsilon}}(e,r)}{Fd\mu_{\bar{g}_{\epsilon}}}=\int_{B_g(o,r)}{fd\mu_g}}. 
 \end{equation}

\noindent Since $G$ is compact and $F$ is smooth, there exists a constant $c$ such that $|F|\leq c$ on $G$. Since $B_{\bar{g}_{\epsilon}}(e,r)\subseteq \pi^{-1}(B_g(o,r))$ (relation \eqref{ClaimBall}), we have

\begin{equation*} \bigg|\frac{1}{v_{\epsilon}}\int_{\pi^{-1}(B_g(o,r))}{Fd\mu_{\bar{g}_{\epsilon}}}-\frac{1}{v_{\epsilon}}\int_{B_{\bar{g}_{\epsilon}}(e,r)}{Fd\mu_{\bar{g}_{\epsilon}}}\bigg|\leq c\bigg|\frac{\mu_{\bar{g}_{\epsilon}}(\pi^{-1}(B_g(o,r)))}{v_{\epsilon}}-\frac{\mu_{\bar{g}_{\epsilon}}(B_{\bar{g}_{\epsilon}}(e,r))}{v_{\epsilon}}\bigg|. \end{equation*}

\noindent Since $F=f\circ \pi$, applying Lemma \ref{FubIN} to the integrable function $F\mathbf{1}_{\pi^{-1}(B_g(o,r))}$ yields

\begin{eqnarray*}
\frac{1}{v_{\epsilon}}\int_{\pi^{-1}(B_g(o,r))}{Fd\mu_{\bar{g}_{\epsilon}}}&=&\frac{1}{v_{\epsilon}}\int_{B_g(o,r)}{\bigg(\int_{\pi^{-1}(p)}{(f\circ\pi)(x)d\mu_{\pi^{-1}(p)}(x)}\bigg)d\mu_g(p)}\\
&=&\int_{B_g(o,r)}{fd\mu_g}.
\end{eqnarray*}
  
  \noindent Therefore, 
  
\begin{equation}\label{KLANT} \bigg|\int_{B_g(o,r)}{fd\mu_g}-\frac{1}{v_{\epsilon}}\int_{B_{\bar{g}_{\epsilon}}(e,r)}{Fd\mu_{\bar{g}_{\epsilon}}}\bigg|\leq c\bigg|\frac{\mu_{\bar{g}_{\epsilon}}(\pi^{-1}(B_g(o,r)))}{v_{\epsilon}}-\frac{\mu_{\bar{g}_{\epsilon}}(B_{\bar{g}_{\epsilon}}(e,r))}{v_{\epsilon}}\bigg|. \end{equation}

\noindent  Applying Equation \eqref{above} for $S=B_g(o,r)$, and using relation \eqref{LimitVolume}, the right-hand side in \eqref{KLANT} tends to zero for $\epsilon \rightarrow 0$. Therefore, relation \eqref{LimitMean1} is true and hence relation \eqref{FirstLimitMean} follows.

Applying the same argument to $F^2=f^2\circ \pi$, we obtain

\begin{equation}\label{LimitMean2}\lim_{\epsilon \rightarrow 0}{\frac{1}{v_{\epsilon}}\int_{B_{\bar{g}_{\epsilon}}(e,r)}{F^2d\mu_{\bar{g}_{\epsilon}}}=\int_{B_g(o,r)}{f^2d\mu_g}}. 
 \end{equation}

\noindent For simplicity, set 

\[ B_{\epsilon}:=B_{\bar{g}_{\epsilon}}(e,r)\ \ \makebox{and} \ \ F_{\epsilon}:=F_{B_{\epsilon}}=\frac{1}{\mu_{\bar{g}_{\epsilon}}(B_{\epsilon})}\int_{B_{\epsilon}}{Fd\mu_{\bar{g}_{\epsilon}}}. \]

\noindent Then we have

\begin{equation*}\int_{B_{\epsilon}}{|F-F_{\epsilon}|^2d\mu_{\bar{g}_{\epsilon}}}= \int_{B_{\epsilon}}{F^2d\mu_{\bar{g}_{\epsilon}}}-2F_{\epsilon}\int_{B_{\epsilon}}{Fd\mu_{\bar{g}_{\epsilon}}}+F_{\epsilon}^2\mu_{\bar{g}_{\epsilon}}(B_{\epsilon})=\int_{B_{\epsilon}}{F^2d\mu_{\bar{g}_{\epsilon}}}-F_{\epsilon}^2\mu_{\bar{g}_{\epsilon}}(B_{\epsilon}). \end{equation*}

\noindent Therefore,

\begin{equation}\label{VarianceIdentity}\frac{1}{v_{\epsilon}}\int_{B_{\epsilon}}{|F-F_{\epsilon}|^2d\mu_{\bar{g}_{\epsilon}}}=\frac{1}{v_{\epsilon}}\int_{B_{\epsilon}}{F^2d\mu_{\bar{g}_{\epsilon}}}-F_{\epsilon}^2\frac{\mu_{\bar{g}_{\epsilon}}(B_{\epsilon})}{v_{\epsilon}}. \end{equation} 

\noindent By virtue of relations \eqref{LimitVolume}, \eqref{FirstLimitMean} and \eqref{LimitMean2}, Equation \eqref{VarianceIdentity} yields

\begin{eqnarray*} \lim_{\epsilon \rightarrow 0}{\frac{1}{v_{\epsilon}}\int_{B_{\epsilon}}{|F-F_{\epsilon}|^2d\mu_{\bar{g}_{\epsilon}}}}&=&\int_{B_g(o,r)}{f^2d\mu_g}-f^2_{B_g(o,r)}\mu_g(B_g(o,r))\\ \\ 
&=&\int_{B_g(o,r)}{|f-f_{B_g(o,r)}|^2d\mu_g}, \end{eqnarray*}

\noindent which proves Claim \ref{ClaimPoincare3}.

Dividing \eqref{PoincareLieGroupsProof} by $v_{\epsilon}$ and letting $\epsilon\rightarrow 0$, Claims \ref{ClaimPoincare1}--\ref{ClaimPoincare3} yield the desired inequality. \qed

\end{document}